\documentclass[a4paper,12pt,reqno]{amsart}
\usepackage[T1]{fontenc}
\usepackage{amsmath,amssymb,latexsym,amsfonts,amsthm,comment, xspace, enumerate} 
\usepackage[english]{babel}
\usepackage{algorithmic, algorithm}
\usepackage[normalem]{ulem} 

\input amssym.def
\input amssym

\usepackage{color}
 
\newtheorem{defn}{Definition}[section]
\newtheorem{lemma}[defn]{Lemma}

\newtheorem{thm}[defn]{Theorem}
\newtheorem{prop}[defn]{Proposition}
\newtheorem{cor}[defn]{Corollary}
\newtheorem{rem}[defn]{Remark}

\newcommand{\Nd} {\mathbb N_0^{n}}

\newcommand{\G}{G^{M_p,A}(\mathbb R^{n}_{+})}
\newcommand{\Ge}{\mathcal G^{\{M_p\}}(\mathbb R^{n}_{+})}

\def\be{\begin{equation}}
\def\ee{\end{equation}}
\def\newin {\kern-0.22em\in\kern-0.15em}
\def\newsubset {\kern-0.2em\subset\kern-0.2em}

\def\<{\langle}
\def\>{\rangle}

\newcommand{\NN}{\mathbb N}

\newcommand{\CC}{\mathbb C}
\newcommand{\RR}{\mathbb R}

\newcommand{\ds}{\displaystyle}

\newcommand{\SSS}{\mathcal S}

\title[Series expansions and solvability of equations] {Ultradistributions    on $\mathbb R_{+}^{n}$. Solvability and hypoellipticity through series expansions of ultradistributions
}

\author[S. Pilipovi\'c]{Stevan Pilipovi\'c}
\address{Stevan Pilipovi\'c, Department of Mathematics and Informatics,
University of Novi Sad, Trg Dositeja Obradovi\'{c}a 4, 21000 Novi Sad, Serbia}
\email{stevan.pilipovic@dmi.uns.ac.rs}

\author[\DJ. Vuckovi\' c]{\DJ or\dj e Vu\v{c}kovi\' c}
\address{\DJ or\dj e Vu\v ckovi\'c, Technical Faculty ``Mihajlo Pupin'', \DJ ure \DJ akovi\'{c}a bb, 23000 Zrenjanin, Serbia}
\email{djordjeplusja@gmail.com}


\begin{document}

\maketitle
\begin{abstract}

In the first part
 we analyze space $\mathcal G^*(\mathbb R^{n}_+)$ and its dual through Laguerre expansions when these spaces correspond to a general sequence $\{M_p\}_{p\in\mathbb N_0}$, where $^*$ is a common notation for  the Beurling and Roumieu cases of spaces. In the second part we are solving equation of the form
$Lu=f,\; L=\sum_{j=1}^ka_jA_j^{h_j}+cE^{d}_y+bP(x,D_x),$  where $f$ belongs to the tensor product of ultradistribution spaces over compact manifolds without boundaries as well as ultradistribution spaces  on $\mathbb R^n_+$ and $\mathbb R^m$; $A_j, j=1,...,k$, $E_y$ and $P(x,D_x)$ are operators  whose eigenfunctions form orthonormal basis of corresponding $L^2-$space. The sequence space representation of solutions enable us to study the solvability and the hypoellipticity in the specified spaces of  ultradistributions. 
\end{abstract}

\section{Introduction}

The contribution of this paper goes in two directions. In the first one, we  analyze the space of ultradifferentiable functions $\mathcal G^{*}$ over $\mathbb R^{n}_+=(0,\infty)^n$ and its dual, the ultradistribution space
$\mathcal G'^{*}(\mathbb R^{n}_+)$, determined by  a general sequence $M_p, p\in\mathbb N_0,$ which satisfies Komatsu  type conditions  (cf. Section \ref{Sec2}).
Such spaces  were analyzed in \cite{JPP} in the case when $M_p=p!^s, s\geq 1, p\in\mathbb N_0$ ($\mathbb N_0=\mathbb N\cup\{0\}$) and the transfer of the analysis in \cite{JPP} to the case which corresponds to a general sequence $\{M_p\}_{p\in\mathbb N_0}$, although the expected one,  is enough complex. Moreover, we impose one new condition $(M.00)$ which implies the series expansions. 

This condition highlights the differences between the results related to the expansion in basic spaces presented in \cite{JPP} and those in this paper.  Both  results are based on \cite{D2}, \cite{D1} which are related to the one-dimensional case (see also \cite{JP} and \cite{JPP1}).

The second direction of our analysis is related to the extension of a generalized functions framework for the solvability and hypoellipticity of equations by the mean of series expansions in involved ultradistribution  spaces. 
We refer to papers \cite{AZ}, \cite{B}-\cite{CC}  for the general questions of hypoellipticity, as well as to \cite{SGK}-\cite{Cappiello}
  where the hypoellipticity of equations is discussed in terms of suitable Diophantine conditions related to the expansions of  distributions and ultradistributions. We will study the equation

 \begin{equation}\label{jedn}
Lu=f,\; L=\sum_{j=1}^ka_jA_j^{h_j}+cE^{d}_y+bP(x,D_x),
\end{equation}
$$f\in
 \otimes_{j=1}^k \mathcal M_j'^*\widehat\otimes
\mathcal G'^*\widehat{\otimes}\mathcal S'^*=
\mathcal H'^*(\prod_{j=1}^kX_j\times\mathbb R_+^{n}\times\mathbb R^{m}),
$$
where ultradistribution spaces are $\mathcal M_j'^*=\mathcal M'^*(X_j)$,  $X_j$ are compact $C^\infty-$manifolds without boundaries, $
\mathcal G'^*=\mathcal G'^*(\mathbb R_+^{n}),$
$\mathcal S'^*= \mathcal S'^*(\mathbb R^{m});$
$a_j, b, c\in\mathbb C,  h_j, d\in \mathbb N_0, j=1,...,k,$
$A_j, j=1,...,k,$ are positive elliptic pseudo-differential operators, 
$$E=-\sum_{i=1}^{n} 
\big(y_i\partial^2/\partial y_i^2 
+ \partial/\partial y_i -y_i/4 +1/2\big), y_i\in(0,\infty), i=1,...,n,$$
 is the Laguerre self-adjoint elliptic operator on $\mathbb R^{n}_+$  and $P(x,D),$ $x\in\mathbb R^m,$ is a normal globally elliptic  Shubin-type operator.

 We  analyse
 (\ref{jedn}) for various values of $a_j, b,  c, h  _j, d,  j=1,...,k,$
  considering the  tensor products	 of ultradistribution spaces over $\mathbb R^{n}_+$,  $\mathbb R^{m}$ and $X_j, j=1,...,k,$ by the use of orthonormal expansions  in all involved spaces. 
  
 We follow  \cite{DR1} and \cite{DR2} for the sequential properties of  spaces  $\mathcal M_j^*=\mathcal M_j^*(X_j), j=1,...,k,$  (with different notation). Elliptic pseudo-differential operators $A_j,
j=1,...,k,$ on $L^2(X_j)$ have real-valued eigenvalues $\lambda_{i}^j, i\in\mathbb N_0,$ (increasing to $\infty$) and eigenfunctions $v_{i}^j\in\mathcal M^j(X_j), i\in\mathbb N_0$, which form orthonormal systems in $L^2(X_j),  j=1,...,k$.

In our investigations  in Section \ref{Sec3}, we consider $\mathcal G^*(\mathbb R^n_+)$ and their duals which correspond to the  Laguerre orthonormal basis  of $L^2(\mathbb R^n_+)$.  Recall,
$\ell_s(y)=\prod_{l=1}^{n}\ell_{s_l}(y_l),$ $ y=  (y_1,...,y_n)\in \mathbb R_{+}^n$ are eigenfunctions of the Laguerre operator $E=E_y$,
 $E^d\ell_s=|s|^d\ell_s$ (eigenvalues are $|s|=s_1+...+s_n,\; s\in\mathbb N^n_0$). The most studied spaces  of ultradifferentiable functions
 $\mathcal S^*(\mathbb R^n)$ and their duals are presented through the series expansions  by eigenfunctions $u_j, j\in\mathbb N_0,$ of a globally elliptic, normal operator $P$ on $L^2(\mathbb R^m)$. These functions   form an orthonormal system in  $L^2(\mathbb R^m).$ The corresponding  eigenvalues are denoted by $\{\mu_i:i\in\mathbb N_0\}$. For the sake of simplicity, we assume that $A_j, j=1,...,k,$ and  $P$ are  non-negative operators so that their eigenvalues are non-negative.

 Spaces $\mathcal M_j^*, j=1,...,k,$  $\mathcal G^*$  and $\mathcal S^*$, and their duals constitute the space of ultradifferentiable functions over $\prod_{j=1}^kX_j\times\mathbb R^n_+\times\mathbb R^m:$
$$\mathcal H^*(\prod_{j=1}^kX_j\times\mathbb R^{n}_+\times\mathbb R^{m})=\otimes_{j=1}^k \mathcal M_j^*\widehat\otimes
\mathcal G^*\widehat{\otimes}\mathcal S^*,$$
where $\widehat \otimes$ denotes the $\pi-$completion of the tensor product. All the spaces are nuclear so that the $\varepsilon-$ and $\pi-$ topology on the tensor products are the same. Its dual $ \mathcal H'^*$ is the framework for solving  (\ref{jedn}). In the case when $X=\mathbb S^{n_1-1}, n=n_2, m=n_3,$ there holds
\begin{equation*}\label{space123}
\mathcal H'^*(\mathbb S^{n_1-1}\times\mathbb R^{n_2}_{+}\times\mathbb R^{n_3})=\mathcal E'^*(\mathbb S^{n_1-1})\widehat\otimes
\mathcal G'^*(\mathbb R^{n_2}_+)\widehat{\otimes}\mathcal S'^*(\mathbb R^{n_3}).
\end{equation*}

As a preparation,  Section \ref{Sec4} is related to the structure of the tensor product spaces. We  show that the expansions do not depend of the order of summations.

We consider $\mathcal M^*(X)$ in the special case when it is  the well-known space of ultradifferentiable functions
$\mathcal E^*(\mathbb S^{n-1})$, where $\mathbb S^{n-1}$ is the unit  sphere in $\mathbb R^{n}$, while $A=\Delta_{\mathbb S^{n-1}}$ is  the Laplace-Beltrami operator on the sphere.

 We note that in  \cite{VP},  ultradistribution spaces on  $\mathbb S^{n-1}\times\mathbb R^{m}$ were considered. Here we are dealing with expansions of  the  form
$$f(\theta,y,x)=\sum_{i=0}^{\infty}\sum_{p=0}^\infty\sum_{j=0}^{\infty}\sum_{k=0}^{N_j}a_{k,j,p,i}Y_{k,j}(\theta)\ell_p(y) h_i(x),$$
 where the global parts (in $x$ and $y$)  are expanded via  Hermite  and Laguerre orthonormal basis  while  the  spherical part is expanded via  spherical harmonics $Y_{k,j}, j\in\mathbb N_0, k\leq N_j.$ Actually, we will use a more general version of orthonormal basis in $L^2(\mathbb R^m).$  Articles  \cite{GPR} and   \cite{VV2} generalize results regarding characterization of Gelfand-Shilov type spaces $\mathcal S'^*$ via Hermite function expansion in order to be applicable for every eigenfunction expansion associated with a normal globally elliptic Shubin-type operator over $\mathbb R^{m}$ of the form
$P=\sum_{|\alpha|+|\beta|\leq q} c_{\alpha\beta} x^{\beta} D^{\alpha},$ where $D^{\alpha}=-i^{|\alpha|}\prod_{i=1}^m\partial_{x_i}^{\alpha_i}.
$
\\ Concerning Laguerre expansions, we accomodate the notation by the renumbering of eigenvalues for operator $E$.

 We analyse  (\ref{jedn})  in Section \ref{Sec5}. At first,
 operators  of the form
\begin{equation}\label{123}
L_{3}=c_1\Delta_{\mathbb S^{n-1}}^{h}+c_2E^d+c_3P, h,d\in\mathbb N, c_1,c_2, c_3\in\mathbb R.
\end{equation}
through the infinite system of coefficient equations. In  Section \ref{Sec5.1}, in the general case we do the same for equation (\ref{jedn}), paraphrasing the aforementioned one.

\section{Spaces}\label{Sec2}

Following \cite{Komatsu}, we fix  a sequence of positive numbers  $M_p,\; p\in\mathbb N_0$, with $M_0=M_1=1$, which satisfies:
\begin{itemize}

\item [$(M.1)\:$] $M^{2}_{p}\leq M_{p-1}M_{p+1},$  $p\in \mathbb N$.
\item [$(M.2)\:$] $ \displaystyle M_{p}\leq A H^p\min_{1\leq q\leq p} \{M_{q} M_{p-q}\},$ $p\in\mathbb{N}$,  $\exists A>0$, $\exists H\geq 1$.
\item [$(M.0)\:$]
$\sqrt{p!}\leq C_{l} l^{p}M_{p}, \;  p\in\mathbb{N}_{0},
$$$ (\mbox{Roumieu case: } \exists  l>0, \exists C_{l}>0)\;
(\mbox{Beurling case: } \forall l>0, \exists C_{l}>0).
$$
 \item[$(M.0^s)$]
$p!\leq A_0H_0^pM_{p}, \;  p\in\mathbb{N}_{0},$
$$\mbox{(Roumieu case } \exists H_0>1, \exists A_0>0)\; (\mbox{Beurling case }\forall H_0>1, \exists A_0>0).
$$ 
 \item[$(M.00)$]
$\exists m_0\in\mathbb N$
$$\forall p>m>m_0, (p\in\mathbb N)\; \;  M_p\geq M_mm^{p-m}.
$$ 
\end{itemize}

\begin{rem}\label{vaz}
Condition $(M.0)$, clearly weaker than $(M.0^s)$ ($s$  stands for  "strong"),
 ensures non-triviality of the  spaces  of test  functions $\mathcal S^*$, while $(M.0^s)$
is enough for the non-triviality of $\mathcal G^*$ and for the test spaces on manifolds. Note that  $(M.0^s)$ also allows the analysis of certain quasi-analytic test function spaces.

In the sequel, we will assume $(M.1), (M.2)$ and $(M.0^s)$ when dealing with  $\mathcal G'^*$ and  ultradistribution spaces on manifolds, and $(M.1), (M.2)$ and $(M.0)$ when dealing with $\mathcal S^*$ type spaces.

The well-known Gevrey sequence  $M_p=p!^s$ satisfies $(M.0)$, if $s\geq 1/2 $ in Roumieu case, while in the Beurling case one must have $s>1/2$. Clearly, $p!^s$ satisfies  $(M.0^s)$ if
$s\geq 1$ in Roumieu case, while in the Beurling case one must have $s>1$.

Condition $(M.00)$ is a new one. It holds for $M_p=p!^\alpha, p\in\mathbb N, \;\alpha \geq 1, $ but does not hold for $M_p=p!/a^p, p\in\mathbb N,$ if $a>1.$
(For $a<1$, it holds.) Moreover, this condition is a consequence of the non-quasi-analytic condition $(M.3)'.$ Namely, by \cite{Komatsu}, 
$(M.3)'$ implies $$pM_p/M_{p+1}\rightarrow 0, p\rightarrow\infty,$$
which implies  $(M.00).$
\end{rem}

In the sequel, when we refer to $^*$, we always first consider Roumieu case, follow by the  Beurling counterpart.

For $\alpha\in\Nd$, we define $M_{\alpha}:=M_{|\alpha|}$.
The associated function of $M_p$ is
 \begin{equation}\label{assf}
M(t)=\sup_{p\in\mathbb{N}_0}\log_+ t^p/M_p,\; t>0.
\end{equation}

  Structural properties of  spaces  $\mathcal M^*(X)$ and  $\mathcal M'^*(X)$
(cf.  \cite{DR1}, \cite{DR2})
 are determined by the positive elliptic  operator $A$ of order $\nu\in\mathbb N$ on a $C^\infty$-  compact manifold $X$ without boundary. It is noted in \cite{DR1} that any other such operator of the same order defines the same space which we  present below.
 Eigenvalues of $A$ are non-negative real  numbers $\lambda_i\leq \lambda_{i+1}, i\in \mathbb N$, and $v_i, i\in\mathbb N_0,$ are eigenfunctions of $A$. 

The following characterization of
$\phi(t)=\sum_{i=0}^\infty{a_iv_i(t)}\in L^2(X), a_i\in\mathbb C, i\in\mathbb N_0,$ is given in \cite{DR1}:
$\phi\in \mathcal M^{\{M_p\}}(X),\mbox{ respectively, } \mathcal M^{(M_p)}(X),$

\noindent if and only if one of the following equivalent conditions holds:
\begin{itemize}
\item[$i)$]
$$\exists h>0, \mbox{ respectively, }\forall h>0, \exists C=C_h>0;$$
$$ ||A^i\phi||_{L^\infty}\leq Ch^{\nu i}M_{\nu i},\; i\in\mathbb N_0;
$$
\item[$ii)$] $\exists h>0,\mbox{ respectively, } \forall h>0,$
$$\sup_{\alpha\in\mathbb N_0^n,x\in X}\frac{|\phi^{(\alpha)}(x)|}{h^{|\alpha|}M_{\nu|\alpha|}}<\infty;
$$
\item[$iii)$] $\exists k>0, \mbox{ respectively, }\forall k>0, \exists C=C_k>0,$
$$ |a_i|\leq
Ce^{-M(k\lambda_i^{1/\nu})}, i\in\mathbb N_0.
$$
\end{itemize}
Let $f\in\mathcal M'^*(X)$ be given as  a formal sum
$f=\sum_{i=0}^\infty b_iv_i$.
Then the following conditions are equivalent:
\begin{itemize}
\item[$i)$]
$f\in \mathcal M'^{\{M_p\}}(X),\mbox{ respectively, } \mathcal M'^{(M_p)}(X);$

\item[$ii)$]
$\forall h>0, \mbox{ respectively, }\exists h>0, \exists C=C_h>0,
$
$$ |\langle f,\phi\rangle|\leq C\sup_{\alpha\in\mathbb N_0^n,x\in X}\frac{|\phi^\alpha(x)|}{h^{|\alpha|}M_{\nu|\alpha|}}, \phi \in
\mathcal M^{\{M_p\}}(X),\mbox{ respect., } \phi\in\mathcal M^{(M_p)}(X);
$$
\item[$iii)$]
$
\forall k>0, \mbox{ respect., }\exists k>0, \exists C=C_k, |b_i|\leq Ce^{M(k\lambda_i^{1/\nu})},\; i\in\mathbb N_0.
$
\end{itemize}
The well-studied space  $\mathcal E^*(\mathbb S^{n-1})$ (cf. \cite{VV1}) serves as an example of $\mathcal M^*(X)$. This is the case when
$X=\mathbb S^{n-1}$ and $A=\Delta_{\mathbb S^{n-1}}$ is the Laplace-Beltrami operator. (We assume in the sequel that $n>2$.)

Spherical harmonics of degree $j, \;Y_{j,k}, k=1,...,N_j, j\in\mathbb N_0$, define the space
 $\mathcal H_j(\mathbb S^{n-1})$. Its dimension $N_j$ satisfies
 $
{2j^{n-2}/(n-2)!}< N_j\leq n j^{n-2}, \; j\geq1.
$ See \cite{Axler} for the comprehensive survey concerning harmonic function theory.

Given a  function $\varphi$ on $\mathbb{S}^{n-1}$, its homogeneous extension of order 0 is the function $\varphi^{\upharpoonright}$ defined as $\varphi^{\upharpoonright}(x)=\varphi (x/|x|)$ on $\mathbb{R}^{n}\setminus\{0\}$. The Laplace-Beltrami operator is defined by $\Delta_{\mathbb S^{n-1}}\phi=\Delta \varphi^{\upharpoonright}(x)$ ($\Delta$ is Laplacian in $\mathbb R^{n}$).
Recall,
\begin{equation*}\label{BL}
\Delta_{\mathbb S^{n-1}}Y_{k,j}=j(j+n-2)Y_{k,j},\;  j\in \mathbb N_0,\; k\leq N_j.\end{equation*}

Then, $\varphi\in\mathcal{E}^*(\mathbb{S}^{n-1})$ if and only if $\varphi^{\upharpoonright}\in\mathcal{E}^*(\mathbb{R}^{n}\setminus\{0\})$.

With the assumptions $ (M.1), (M.2), (M.0),$ the  Banach space  ${\mathcal S}^{M_p,A}_{L^{2}}(\mathbb R^n)$, $A>0$, consists of all $f\in C^{\infty}(\mathbb{R}^{n})$ such that
\begin{equation*}
\label{ultranorms}
\|f\|_{A} =\sup_{\alpha,\beta\in\mathbb{N}^{n}_0}\frac{\|\langle x\rangle^{\beta}\partial^{\alpha}f\|_{L^{2}(\mathbb{R}^{n})}}{A^{|\alpha|+|\beta|}M_{|\alpha|+|\beta|}}<\infty.
\end{equation*}
 Roumeu, respectively, Beurling spaces of ultradifferentiable functions are defined as inductive, respectively, projective topological  limits
\begin{equation*}
\label{eqspaces1}
{\mathcal{S}}^{\{M_p\}}(\mathbb R^{n})=\bigcup_{A>0} {\mathcal S}^{M_p,A}_{L^{2}}(\mathbb R^n) \quad \mbox{ and} \quad {\mathcal S}^{(M_p)}(\mathbb R^{n})=\bigcap_{A>0} {\mathcal S}^{M_p,A}_{L^{2}}(\mathbb R^n).
\end{equation*}

We note that  the test spaces of Roumieu type are
 (DFS)-spaces while those of Beurling type are (FS)-spaces (cf. \cite{Komatsu}). Both are nuclear spaces.

Concerning  $P,$  a normal globally elliptic differential operator of Shubin type of order $q\in\mathbb N$, eigenvalues $\mu_j$ and eigenfunctions $u_j, j\in\mathbb N_0$, the following is proved in  \cite{GPR}: 
 Let
$
\phi=\sum_{j=0}^{\infty} a_j u_j\in L^{2}(\mathbb{R}^{n}), a_j\in\mathbb C,\; j\in\mathbb N_0.
$
Then,
$\phi\in{\mathcal S}^{\{M_p\}}(\mathbb R^{n})$, respectively, $\phi\in{\mathcal S}^{(M_p)}(\mathbb R^{n}),$ if and only if
$$\exists \lambda>0, \mbox{ respectively, }  \forall \lambda>0, \exists C=C_\lambda>0,
$$
\begin{equation*}
\label{vazno1}
|a_j|\leq C_\lambda e^{-M( \lambda \mu_j^{{1}/{q}})}, \; j\in\mathbb{N}_0.
\end{equation*}

 If $f$ is given by  a formal sum
$
f=\sum_{j=0}^{\infty} b_j u_j,
$
then,
$f\in{\mathcal S}'^{\{M_p\}}(\mathbb R^{n})$, respectively, $f\in{\mathcal S}'^{(M_p)}(\mathbb R^{n}),$ if and only if
$$\forall \lambda>0, \mbox{ respectively, }  \exists \lambda>0, \exists C=C_\lambda>0,
$$
\begin{equation}
\label{vazno11}
|b_j|\leq C_\lambda e^{M( \lambda \mu_j^{{1}/{q}})}, \; j\in\mathbb{N}_0.
\end{equation}

We note that these spaces are a convenient  framework for the \\  time - frequency  analysis, see \cite{T1,T2,RV} and references therein.

\section{$\mathcal G^*$-type spaces}\label{Sec3}

We refer to \cite{JPP} for the properties of the space $\mathcal S(\mathbb R^n_+)$ and its dual space of tempered distributions supported by $\overline{\mathbb R^n_+}$. 

In this section, we assume $ (M.1), (M.2)$ and $(M.0^s)$.  In addition to the consideration in \cite{JPP}, we include Beurling type spaces into our consideration. 
 \begin{defn}
  Let $A>0$. Consider the space
$$ G^{M_p,A}(\mathbb R^n_+)=\{f\in  \mathcal S(\mathbb R^{n}_{+}): \sup_{p,k\in\mathbb N_0^d}\frac{\| t^{\frac{p+k}{2}}D^p f(t)\|_2}{A^{p+k} \sqrt{M_p M_k}}<\infty\},$$
with the seminorms 
\begin{equation*}\label{seminorm}
\sigma_{A,j}(f)=\sup_{p,k\in\NN^d_0}\frac{\|t^{(p+k)/2}D^pf(t)\|_{L^2(\RR^d_+)}}
{A^{|p+k|} \sqrt{M_p M_k}}+\sup_{|p|\leq j, |k|\leq j} \sup_{t\in\RR^d_+}|t^k D^pf(t)|,\,\, j\in\NN_0.
\end{equation*}
Then $ G^{M_p,A}(\mathbb R^n_+)$ is a Fr\' echet space
and the  inductive topological limit (Roumieu case), respectively, projective topological limit (Beurling case),

$$\ds\Ge=\cup_{A>0}\G,\; \mathcal G^{(M_p)}=\cap_{A>0}\G,
$$
are the basic spaces.
 \end{defn}
 They are barrelled and bornological  spaces. The sequence space characterizations of these spaces  given below in the section \ref{Sec3.2} show that
these spaces are also nuclear ones. In order to proceed further on, we need   the next  lemma. It is proved in \cite{Komatsu}, more or less, in the form which we
present below.

 \begin{lemma} (\cite{Komatsu}) \label{koren}

 a) The associated function satisfies
\begin{equation}\label{nejednakostM} M(x+y)\leq M(2x)+M(2y),\; x, y> 0;
\end{equation}

b) Let $\lambda>1$. Then, 
 \begin{equation}\label{nejednakostM2}
\exists C>0, \exists a>0,  \lambda M(t)\leq M(\lambda^a t)+C,\; t>0 ;
 \end{equation}

c)  $\forall p\in\mathbb N, \exists C>0, \exists H>0,\; M_{\lfloor \frac{ p}{2}\rfloor}\leq  \sqrt{M_p}\leq C h^p M_{\lfloor\frac{p}{2}\rfloor}$
\\ ($\lfloor p/2 \rfloor$ denotes the integer part of $p/2$).
\end{lemma}

Also we need the next lemma.

\begin{lemma}\label{suma}
a) Let $m=(m_1,...,m_n)\in\mathbb N^n.$ Then
$$
\sum_{m\in\mathbb N} \Big(\prod_{k=1}^n \frac{1}{m_k^{n+1}}\Big)<\infty.
$$

b) If  the sequence $\{M_p\}$ satisfies $M.00$ then for any $h>1$ ,  $m,p\in\mathbb N, p>m\geq m_0$ (where $m_0$ is the constant from $M.00$) there holds
$$\frac{h^mM_m}{m^m}\leq \frac{h^pM_p}{m^p},\;p>m\geq m_0.
$$ 
\end{lemma}
\begin{proof}
a) We start with the obvious inequality $$m_1+m_2+\cdots +m_n\leq n\cdot m_1m_2\cdots m_n.$$
Therefore, $\prod_{k=1}^n m_n^{-n-1}\leq n^{n+1} (m_1+\cdots m_n)^{-n-1}$.
This implies
$$\sum_{m\in\Lambda} \Big(\prod_{k=1}^n \frac{1}{m_k^{n+1}}\Big)\leq \sum_{p=n}^{\infty} C_p n^{n+1} p^{-n-1},  $$
where $C_p={card} (\{m\in\mathbb N^n: |m|=p\})={p-1 \choose n-1}$. Here, ${card}(A)$ denotes the cardinal number of the set $A$.

 Moreover,
$$
n^{n+1} p^{-n-1} {p-1 \choose n-1}\leq \frac{n^{n+1}}{(n-1)!}  \Big( \prod_{l=0}^{n-2}  \frac{p-1-l}{p}              \Big) \cdot \frac{1}{p^2}
$$
and due to the fact that $\sum_{p=n}^{\infty} {1}/{p^2} <\infty$, the claim follows.

b) By $(M.00)$, for any $h>1,$
$$h^mm^{p-m}M_m\leq h^p M_p\; \Rightarrow \;h^pM_p/m^p\geq h^mM_m/m^m, \; p>m, p,m\in\mathbb N.
$$
\end{proof}

\subsection{On the modified Hankel Clifford transforms. Recapitulation}\label{HC}\label{Sec3.1}

 In this section we follow results of Duran \cite{D2}, \cite{D1} concerning the Henkel-Clifford transform(in case $n=1$) which are
carefully presented and extended in Section 4 of \cite{JPP}  to the Henkel-Clifford transform in case $n>1$, in the Roumieu type space of test functions  $\mathcal G^\alpha_\alpha$  ($M_p=p!^\alpha, \; p\in\mathbb N,  \alpha\geq 1$).
For the sake of completeness, we recall some necessary results for our exposition for the general $\{M_p\}_{p\in\mathbb N_0}$. We restrict this repetition to the case 
$\gamma=0$, although in quoted papers the analysis is given for $\gamma\geq 0$ and Laguerre functions 
$\mathcal L^\gamma_p, p\in\mathbb N_0^n$ (this is notation from \cite{JPP}, here, for 
$\gamma=0$, we use $\ell_p$). Recall, $ J_0$ and $ I_0$ are Bessel function of the first kind and modified Bessel function of the first kind, respectively. Let   $z\in\mathbf{T}^{(n)}=\{z\in\CC^n|\, |z_l|=1,\, z_l\neq 1,\, 
l=1,\ldots,n\}$. The fractional power and modified fractional power of the Hankel-Clifford transform of an $f\in\mathcal S(\mathbb R^n_+)$ are defined by
$$\mathcal{I}_{z,0}f(t)=\left(\prod_{l=1}^n(1-z_l)^{-1}e^{-\frac{1}{2}\frac{1+z_l}{1-z_l}t_l}\right)
\int_{\RR^n_+}f(x)\prod_{l=1}^n
e^{-\frac{1}{2}
\frac{1+z_l}{1-z_l}x_l}I_{0}\left(\frac{2\sqrt{x_lt_lz_l}}{1-z_l}\right)dx,
$$
$$
\mathcal{J}_{z,0}f(t)=\left(\prod_{l=1}^n(1-z_l)^{-1}\right)\int_{\RR^n_+}f(x) \prod_{l=1}^nI_{0}\left(\frac{2\sqrt{x_lt_lz_l}}{1-z_l}\right)dx,
$$
respectively.  Proposition 4.2 in \cite{JPP} gives that  $\mathcal{I}_{z,0}$ and $\mathcal{J}_{z,0},\; z\in\mathbb T^{(n)}$ are topological isomorphisms on $\mathcal S(\mathbb R^n_+)$ which have isometric extensions on $L^2(\mathbb R^n_+)$ with the inverse $\mathcal{I}_{\overline z,0}$ and $\mathcal{J}_{\overline z,0}.$
The key identity is  the one given by Duran
 \cite[Lemma 3.2]{D2} in case $n=1$, and extended to the case $n>1$ in \cite{JPP}:\\
\begin{equation*}
\left\|t^{(p+k)/2}D^pf(t)\right\|_2=\left(\prod_{l=1}^d|1-z_l|^{-p_l+k_l}\right)
\left\|t^{(p+k)/2}D^k\mathcal{J}_{z,0}f(t)\right\|_2,
\end{equation*}
$$f\in
\SSS(\mathbb{R}^d_+),\; p, k\in\mathbb{N}_0^n.$$
A consequence of this identity is that
\begin{itemize}
\item[]
$\;\;\;\;\;\;\;\;\;  \mathcal{J}_{z,0}$ is an isometry over $L^2(\mathbb R^n_+)$  and
\item[]\begin{equation}\label{izom}
\mathcal{J}_{z,0} \mbox{ is an isomorphism of } \mathcal G^*(\mathbb R^n) \mbox{ onto } \mathcal G^*(\mathbb R^n_+).
\end{equation}
\end{itemize}
If $z=-\bold 1,$ Then  $\mathcal{J}_{z,0}$ is denoted as $\mathcal H_0$. This is the $n-$dimensional Hankel-Clifford transform. (We will use the same notation for the one-dimensional case.)

Next, in \cite{JPP} is considered modified fractional power of the partial Hankel-Clifford transform.
Define, (with $\bold 0'\in\mathbb R^{n'}$)
$$ \mathcal{J}^{(n')}_{z',\bold 0'}f(t)=
\left(\prod_{l=1}^{n'}(1-z_l)^{-1}\right)\times
$$
\begin{equation*}\label{jkomp}\int_{\RR^{n'}_+}f(x',t'')\prod_{l=1}^{n'}
I_{0'}
\left(\frac{2\sqrt{x_lt_lz_l}}{1-z_l}\right)dx', f\in\mathcal S(\mathbb R^n_+),\end{equation*}
where $x=(x',x'')\in\RR^n$, $x'=(x_1,...,x_{n'})$ and
$x''=(x_{n'+1},\ldots,x_{n})$ (similarly for $t=(t',t'')$).
Note that 
$\mathcal{J}^{(n')}_{z',\bold 0'}f=(\mathcal{J}^{n'}_{z',\bold 0'})\hat{\otimes}(Id^{n''})f$ for
$f\in\SSS(\RR^{n'}_+)\otimes\SSS(\RR^{n''}_+).$
Moreover, it is proved that
 $$\left\|t'^{(p'+k')/2}t''^{(p''+k'')/2}D^{p'}_{t'}f(t)\right\|_2=
$$
\begin{equation*}
\left(\prod_{l=1}^{d'}|1-z_l|^{-p_l+k_l}\right)
\left\|t'^{(p'+k')/2}t''^{(p''+k'')/2}D^{k'}_{t'}
\mathcal{J}^{(n')}_{z',\bold 0'}f(t)\right\|_2,\; f\in\mathcal S(\mathbb R^n_+).
\end{equation*}

Let $\Lambda'=\{\lambda'_1,\ldots,\lambda'_{n'}\}\subseteq
\{1,\ldots,n\}$ and
$\Lambda''=\{1,\ldots,n\}\backslash
\Lambda',$
$x_{\Lambda'}=(x_{\lambda'_1},\ldots,x_{\lambda'_{n'}})$ and
abusing the notation we write $x=(x_{\Lambda'},x_{\Lambda''})$ (similarly for $(t=(t_{\Lambda'},t_{\Lambda''})$).
The modified fractional power of
the partial Hankel-Clifford transform with respect to $\Lambda'$ is given in \cite{JPP} by
$$\mathcal{J}^{(\Lambda')}_{z',\bold 0'}f(t)=
\left(\prod_{l=1}^{n'}(1-z_l)^{-1}\right)\times$$$$
\int_{\RR^{d'}_+}f(x_{\Lambda'},t_{\Lambda''})
 \prod_{l=1}^{n'}I_{\bold 0'}
\left(\frac{2\sqrt{x_{\lambda'_l}t_{\lambda'_l}z_l}}{1-z_l}\right)
dx_{\Lambda'}.$$
Then, (\cite{JPP}) $\mathcal{J}^{(\Lambda')}_{\bar{z'},\bold 0'}$ is its inverse. Again, note,  that 
$\mathcal{J}^{(\Lambda_{n'})}_{\overline z',\bold 0'}f=(\mathcal{J}^{\Lambda_{n'}}_{\overline z',\bold 0'})\hat{\otimes}(Id^{\Lambda_{n''}})f$ for
$f\in\SSS(\RR^{n'}_+)\otimes\SSS(\RR^{n''}_+).$  Let $f\in\SSS(\RR^n_+)$ and  $(p_{\Lambda'},p_{\Lambda''}), (k_{\Lambda'},k_{\Lambda''})\in\NN^n_0$. Then,   there holds
$$\left\|t_{\Lambda'}^{(p_{\Lambda'}+k_{\Lambda'})/2}
t_{\Lambda''}^{(p_{\Lambda''}+k_{\Lambda''})/2}D^p_tf(t)\right\|_2=
$$ \begin{equation*}\label{forrandomin}
\left(\prod_{l=1}^{d'}|1-z_l|^{-p_{\lambda'_l}+k_{\lambda'_l}}\right)
\left\|t_{\Lambda'}^{(p_{\Lambda'}+k_{\Lambda'})/2}
t_{\Lambda''}^{(p_{\Lambda''}+k_{\Lambda''})/2}
D^{k_{\Lambda'}}_{t_{\Lambda'}} D^{p_{\Lambda''}}_{t_{\Lambda''}}
\mathcal{J}^{(\Lambda')}_{z',\bold 0'}f(t)\right\|_2.
\end{equation*}
This gives that 
\begin{itemize}
\item[]
$\;\;\;\;\;\;\;\;\;  \mathcal{J}^{(\Lambda')}_{z',\bold 0'}$ is an isometry over $L^2(\mathbb R^n_+)$  and
\item[] \begin{equation}\label{izom2}
\mathcal{J}^{(\Lambda')}_{z',\bold 0'} \mbox{ is an isomorphism of } \mathcal G^*(\mathbb R^n_+ ) \mbox{ onto } \mathcal G^*(\mathbb R^n_+).
\end{equation}
\end{itemize}
We will consider  below the case $z'=-\bold 1\in\mathbb T^{n'}$ and the notation  $\mathcal H^{\Lambda'}_0$ for $\mathcal H^{\Lambda'}_{\bold 0'}$. Moreover, 
$\mathcal H^{\Lambda'}_0 \mbox{ is a self inverse of } \mathcal G^*(\mathbb R^n_+) \mbox{ onto } \mathcal G^*(\mathbb R^n_+).
$
Note that 
\begin{equation}\label{exp1}
 f=\sum_{m\in\mathbb N^n}a_m\ell_m\in\mathcal S(\mathbb R^n_+), \; \Longrightarrow
 \mathcal H_0^{\Lambda'}f= \sum_{m\in\mathbb N^n}(-1)^{n_1^{\lambda'}+...+n_1^{\lambda'}}a_m\ell_m.
\end{equation}


\subsection{Series expansions}\label{Sec3.2}

With the quoted assumption on $\{M_p\}_{p\in\mathbb N_0}$ the associated function of the form (\ref{assf}) and the Laguerre orthonormal basis $\ell_i,\; i\in \mathbb N^n_0$ of 
$L^2(\mathbb R^n_+)$, we have the following version of \cite[Prop. $5.3$]{JPP}).
\begin{prop} \label{5.1}  Let $f\in L^2(\mathbb R_{+}^{n})$ and,  $$a_m=\int_{\mathbb R^{n}_{+}} f(t)\ell_m(t)dt,\; m\in\mathbb N_0^{n}.$$
Assume, $$\exists h>0, \mbox{ respectively, } \forall h>0, \exists C=C_h>0,$$
\begin{equation}\label{uslovopadanja}
|a_m|\leq C e^{-M(h|m|)}, m\in\mathbb N^n_0.
\end{equation}
Then, $f\in\mathcal G^{\{M_p\}}(\mathbb R^n_+)$, respectively, $f\in \mathcal G^{(M_p)}(\mathbb R^n_+)$.
\end{prop}
\begin{proof}
We  give the proof in the Roumieu case  accommodating the proof from \cite{JPP} to the general sequence $\{M_p\}_{p\in\mathbb N_0}$.
The proof in the Beurling case is similar.

Let $m=(m_1,\dots, m_n)\in\mathbb N^n$. Cases when some of coordinates of $m$ equals zero should be separately considered but we skip this part since it can be treated in a similar way.  Let  $t\geq 0$. Since $M$ is an increasing function, due to (\ref{nejednakostM2}), there exist $C>0$ and $a>0$ such that
$$M(t m_1)+M(t m_2)+\dots M(t m_n)\leq n M(t |m|)\leq M(n^a t|m|)+C.$$
 Therefore, with $h_1=h/n^a,$
$$-M(h|m|)\leq C - M(h_1 m_1)-M(h_1 m_2)-\cdots - M(h_1 m_n), \;m\in\mathbb N^n. $$
Suppose now that (\ref{uslovopadanja}) holds.
Then there exists another $C>0$ such that$$\displaystyle |a_m|\leq C\prod_{k=1}^{n} e^{-M(h_1 m_k)}.$$

Let $p=(p_1,p_2,\dots, p_n)\in \mathbb N_0^n$ and $\Lambda\subseteq \{1,2,\dots,n\}$.
By the use of (\ref{uslovopadanja}), the
estimates
  \begin{equation*}\label{Nej1}
\left\|t^{p/2}\ell_m(t)\right\|_2\leq
2^{|p|+5n}\prod_{k=1}^n (m_k+1)\ldots\left(m_k+\lfloor\frac{p_k}{2}\rfloor+2\right),
\end{equation*}
\begin{equation*}\label{Nej2}
\left\|t^{p/2}D^p\ell_m(t)\right\|_2\leq
2^{|p|+ 5n}\prod_{k=1}^n
(m_k+1)\ldots\left(m_k+\lfloor\frac{p_k}{2}\rfloor+2\right), p\in\mathbb N_0^n,
\end{equation*}
given in  the proof of
\cite[Lemma 2.1]{D1} (in case $n=1$)
(estimates (2.5) and (2.6) there),  and (\ref{izom}, \ref{izom2}) imply
that   there exists $C>0$ such that
\begin{equation}\label{ocD}\|t^{\frac{p}{2}} \mathcal H_0^{(\Lambda)} f(t)\|_2\leq C2^{|p|}\sum_{n\in\mathbb N_0^{n}}\prod_{k=1}^{n} e^{-M(h_1 m_k)} (m_k+1)\cdots (m_k+\lfloor\frac{p_k}{2}\rfloor+2), \;m\in\mathbb N^n. 
\end{equation}

We know (see e.g. \cite{Komatsu}) that
$2M(\rho)\leq M(H\rho)+\log(A), \rho >0 $
(constants $A,H>0$ are from $(M.2)$). Let $H\rho=t.$ Then, with $h_2=h_1/H,$ one has
$ 2M({t}/{H})\leq M(t)+\log A$ and $$e^{-M(h_1m_k)}\leq A e^{-2M(h_2m_k)},\; 
k\in\{1,...,n\}.$$
In order to estimate products of the form
\begin{equation}\label{proizvod}2^{|p|} e^{-M(h_1 m_k)} (m_k+1)\cdots (m_k+\lfloor\frac{p_k}{2}\rfloor+2),\end{equation}

first we estimate the product \\
$$e^{-M(h_2 m_k)}(m_k+1)(m_k+2)\cdots (m_k+\lfloor \frac{p_k}{2}\rfloor+2).$$

By (\ref{nejednakostM}), $M(x)\geq M((x+y)/2)-M(y),\; x,y>0$. Therefore, with $y=h_2l, l>0,$ and $x=h_2m_k$, one has 
\begin{equation}\label{nejM}-M(h_2 m_k)\leq - M(\frac{h_2}{2} (m_k+l))+M(h_2 l),\; l>0.\end{equation}
 Fix $l=\lfloor \frac{p_k}{2}\rfloor+2$. Then, by (\ref{nejM}),
$$e^{-M(h_2 m_k)}(m_k+1)(m_k+2)\cdots (m_k+\lfloor \frac{p_k}{2}\rfloor+2)\leq$$
\begin{equation*}\label{umetak}e^{-M(h_2 m_k)} (m_k+\lfloor \frac{p_k}{2}\rfloor+2)^{\lfloor \frac{p_k}{2}\rfloor+2}
\leq \frac{(m_k+\lfloor \frac{p_k}{2}\rfloor+2)^{\lfloor \frac{p_k}{2}\rfloor+2}}
{e^{M\big(\frac{h_2}{2}(m_k+\lfloor \frac{p_k}{2}\rfloor+2)\big) }}e^{M\big(h_2 (\lfloor \frac{p_k}{2}\rfloor+2)\big)}.
\end{equation*}

By   definition of the associated function, we have that $\displaystyle M_p\geq {\rho^p}/{e^{M(\rho)}}, \; \rho>0, p\in\mathbb N$. Therefore,

$$ \frac{(m_k+\lfloor \frac{p_k}{2}\rfloor+2)^{\lfloor \frac{p_k}{2}\rfloor+2}}
{e^{M(\frac{h_2}{2}(m_k+\lfloor \frac{p_k}{2}\rfloor+2)) }}\leq \Big(\frac{2}{h_2}\Big)^{\lfloor \frac{p_k}{2}\rfloor+2} M_{\lfloor \frac{p_k}{2}\rfloor+2}.$$

Denote by $M_{p!}(\cdot)$  the associated function with respect to the sequence $M_p=p!$. The first inequality  in (\ref{jos}) which is to follow, follows from $(M.0^s)$
and the second one from   the fact that $M_{p!}(\rho)$ behaves as $c \rho, \rho\rightarrow\infty $, with suitable $c>0$. Thus,  with the  new constants $L, L_1>0$ and $C>0$, there holds
\begin{equation}\label{jos}M\big( h_2 (\lfloor \frac{p_k}{2}\rfloor+2)\big)\leq L_1M_{p!}\big(h_2(\lfloor \frac{p_k}{2}\rfloor+2)\big)
+C\leq L(\lfloor \frac{p_k}{2}\rfloor+2)+C. 
\end{equation}
So, we estimate the product of the form (\ref{proizvod}), with appropriate new constants $C$ and $L$  as follows:

$$
 2^{p_k}    e^{-M(h_1 m_k)} (m_k+1)\cdots (m_k+\lfloor\frac{p_k}{2}\rfloor+2)
$$
$$\leq   2^{p_k}C  e^{-M(h_2 m_k)} L^{p_k}  \Big(\frac{2}{h_2}\Big)^{\lfloor \frac{p_k}{2}\rfloor+2} M_{\lfloor \frac{p_k}{2}\rfloor+2}.
$$
Due to $(M.2)$, $$M_{\lfloor \frac{p_k}{2}\rfloor+2}\leq A H^{\lfloor \frac{p_k}{2}\rfloor+2} M_2 M_{\lfloor \frac{p_k}{2}\rfloor}.$$
 On the other hand,  using Lemma \ref{koren} c), we infer, with new $C$ and $h$,
$$ 2^{p_k}    e^{-M(h_1 m_k)} (m_k+1)\cdots (m_k+\lfloor\frac{p_k}{2}\rfloor+2)\leq  C  h^{p_k} e^{-M(h_2 m_k)} \sqrt{M_{p_k}},.$$
By  definition of $M$, one has,
 $$e^{-M( h_2 m_k)}\leq h_2^{n+1}m_k^{-n-1}M_{n+1},\; n\in\mathbb N_0.$$
 With this, we continue to estimate (\ref{ocD}),
$$\|t^{\frac{p}{2}} \mathcal H_0^{(\Lambda)} f(t)\|_2 \leq  C^n  h_2^{(n+1)n}M_{n+1}^{n} h^{|p|}\sqrt{\prod_{k=1}^{n} M_{p_k}}\cdot \sum_{m\in \mathbb N_0^n} \frac{1}{m_1^{-n-1}\cdots m_n^{-n-1}}.  $$
Since the latter sum converges (from Lemma  \ref{koren}), with new $C$ and $h$, there holds
\begin{equation}\label{1e1}\|t^{\frac{p}{2}} \mathcal H_0^{(\Lambda)} f(t)\|_2 \leq  C h^{|p|}\sqrt{M_{|p|} },\ p\in \Nd\end{equation}
The same calculation shows (with  constants $\tilde C$ and $\tilde h$),
\begin{equation}\label{2e2}\|t^{\frac{p}{2}} D^p\mathcal H_0^{(\Lambda)} f(t)\|_2 \leq  \tilde C \tilde h^{|p|}\sqrt{M_{|p|} },\ p\in \Nd.\end{equation}
Moreover,
$$(t^{\frac{p+k}{2}} D^p \mathcal H_0^{(\Lambda) }f(t), t^{\frac{p+k}{2}} D^p\mathcal H_0^{(\Lambda) }f(t))$$$$=\Big(D^p\big(t^{p+k} D^p \mathcal H_0^{(\Lambda)} f(t)\big),\mathcal H_0^{(\Lambda)} f(t)\Big),\; p, k\in \mathbb N_0^n.$$
When $2k\geq p$, we have
$$\|t^{\frac{p+k}{2}} D^p \mathcal H_0^{(\Lambda)} f(t)\|_2^2\leq $$$$\sum_{m\leq p} {p\choose m} \frac{(p+k)!}{(p+k-m)!} \cdot\Big|\Big(t^{p+k-m}D^{2p-m} \mathcal H_0^{(\Lambda)} f(t),\mathcal H_0^{(\Lambda)} f(t)\Big)\Big|
$$

$$\leq 2^{|p+k|}\sum_{m\leq p}{p \choose m} m! \cdot  \Big|\Big(t^{\frac{2p-m}{2}}D^{2p-m} \mathcal H_0^{(\Lambda)} f(t),t^{\frac{2k-m}{2}}\mathcal H_0^{(\Lambda)} f(t)\Big)\Big|
$$
Because of $(M.0^s),$ (\ref{1e1}), (\ref{2e2}), (and relations given in Section \ref{HC}) we have

$$
m!  \Big|\Big(t^{\frac{2p-m}{2}}D^{2p-m} \mathcal H_0^{(\Lambda)} f(t),t^{\frac{2k-m}{2}}\mathcal H_0^{(\Lambda)} f(t)\Big)\Big|\leq $$
$$\leq C_0 A^m M_m \|t^{\frac{2p-m}{2}}D^{2p-m} \mathcal H_0^{(\Lambda)} f(t)\|_2 \cdot \| t^{\frac{2k-m}{2}}\mathcal H_0^{(\Lambda)} f(t)\|_2
$$
$$
\leq C_0A^m C\tilde C(h\tilde h)^{|p+k|} M_m\cdot\sqrt{M_{2p-m} M_{2k-m}}.
$$

Note that $2^{|p+k|}\sum_{m\leq p} {p\choose m} =2^{|2p+m|}$. Therefore,  with new constants $C, C_1, C_2, h, h_1, h_2 > 0$:

$$ \|t^{\frac{p+k}{2}} D^p \mathcal H_0^{(\Lambda)} f(t)\|_2^2\leq
 C h^{|p+k|} \max_{m\leq p} \{ M_m\cdot\sqrt{M_{2p-m} M_{2k-m}} \} $$
 $$\leq C_{1} h_{1}^{|p+k|} \max_{m\leq p}\{\sqrt{M_{2m}\cdot M_{2p+2k-2m}}\}\leq C_{2}h_{2}^{|p+k|} M_{p+k}.
$$
In the last set of inequalities we used Lemma 3.2, c) and  the property $(M.1)$  for
 $$M_{2m} M_{2p+2k-2m}\leq M_{2p+2k}.$$

We finally conclude that there exist (new) constants $C>0$ and $h>0$ such that for every $\Lambda\subseteq \{1,2,\dots, n\}$, $k,p\in\Nd, 2k\leq p,$ there holds
\begin{equation}\label{nejednakost5.6}
\|t^{\frac{p+k}{2}} D^p \mathcal H_0^{(\Lambda)} f(t)\|_2\leq  C h^{|p+q|} \sqrt{M_{p+q}}.
 \end{equation}
Now, let $k,p\in \mathbb N_0^n$ be arbitrary chosen. Let $\Lambda'=\{\lambda_1,\lambda_2, \dots, \lambda_{n'}\}\subseteq \{1,2,\dots, n\}$ such that $k_{\lambda'_l}<
{p_{\lambda'_l}}/{2}$ for every $l\in\{1,2,\dots, n'\}$ and let $\Lambda''=\{\lambda''_1,\dots, \lambda''_{n''}\}$ such that $k_{\lambda''_l}\geq {p_{\lambda''_l}}/{2}$ for every $l\in\{1,2,\dots, n''\}.$ Then (\ref{nejednakost5.6}) and (\ref{izom2}) imply  that there exist  constants $C, h>0$  such that
\begin{equation*}
\|t^{\frac{p+k}{2}} D^p f(t) \|_2\leq 2^{|k|} \|t^{\frac{p+k}{2}} D^{k_{\Lambda'}}_{t_{\Lambda'}}D^{p_{\Lambda''}}_{t_{\Lambda''}}\mathcal H^{(\Lambda')} f(t)\|_2\leq C h^{|p+q|} \sqrt{M_{p+q}}.
\end{equation*}
This proves that  $f\in \mathcal G^*.$
\end{proof}
In order to prove that $f\in\mathcal{G}^*(\mathbb R^{n}_+)$ implies (\ref{uslovopadanja}), we repeat the necessary notation from \cite{JPP}. Let $\mathbf\Pi=\Pi_1\times\Pi_2\times\dots\times \Pi_{n},$
where $$\Pi_l=\{z\in \mathbb C: {\rm Im}(z)<0\}, \;l=1,2\dots,n.$$ 
Clearly, for  each $z=x+iy\in\mathbf \Pi$,
the function $t\mapsto e^{-2\pi i\cdot zt}:\mathbb R_{+}^{n}\rightarrow \mathbb C, $ is in $\mathcal G^*(\mathbb R^n)$. Then, define
$$\mathbf \Pi\rightarrow \mathbb C,\; \Pi \ni z\mapsto \mathcal F_{\mathbf \Pi}u(z)=\langle u(t),e^{-2\pi izt}\rangle.
$$
It is a holomorphic function in $\mathbf \Pi.$ The proof is the standard one.
Let $\mathbf D=D_1\times\dots\times D_n$, be a product of complex unit discs. We now repeat from \cite{JPP} the explanation for the definition of
$\mathcal F_{\mathbf D}u(\omega),\; \omega\in\mathbf D, u\in \mathcal G'^*(\mathbb R^n_+):$

$$\Omega(\omega)=(\frac{1+\omega_1}{4\pi i(1-\omega_1)}, \dots,\frac{1+\omega_{n}}{4\pi i(1-\omega_{n})}),\; \omega\in\mathbf D, $$
is a biholomorphic mapping from $\mathbf D$ onto $\mathbf\Pi$, and the inverse mapping is given by
$$\Omega^{-1}(z)=(\frac{4\pi i z_1-1}{4\pi i z_1+1},\dots, \frac{4\pi i z_{n}-1}{4\pi i z_{n}+1}),\; z\in \Pi.$$

For $u\in \mathcal G'^*(\mathbb R^n_+) $, $\mathcal F_D u(\omega)$  is defined by 
$$\mathcal F_D u(\omega)=\mathcal F_{\mathbf \Pi}u(\Omega(\omega))=\langle u(t), \prod_{l=1}^{n} e^{-\frac{1}{2} \frac{1+\omega_l}{1-\omega_l}t_l }\rangle, \; z\in \mathbf D.$$
So, $\mathcal F_D u$ is a holomorphic function in  $\mathbf D$, as it is stated in Lemma 5.2 of \cite{JPP}.

Furthermore, if $u\in ({\mathcal S(\RR^d)})'$ 
then 

\begin{equation}
\ds{\mathcal F}_D(u)(\omega)=\prod_{j=1}^d (1-\omega_j)\cdot \sum_{m\in\NN^d_0} a_m(u) \omega^m, \omega\in \mathbf D
\end{equation}
where $\{a_n(u):n\in \NN^d\}$ is the Laguerre eigenexpansion of the function $u$, see \cite[Prop. 5.3]{JPP}.

 We continue with the same idea as in  \cite[Prop. $5.4$]{JPP}.
\begin{prop}  \label{5.4}  Let $\{a_m\}_{m\in \mathbb N_0^n}$ be a sequence of complex numbers so that  $\displaystyle\lim_{|m|\to\infty}a_m= 0$. Let
$$F(\omega)=(1-\omega)^{\mathbf  1} \sum_{m\in\mathbb N_0^{n}} a_m\omega^m, \omega\in\mathbf D, \;(\omega^m=\omega_1^{m_1}\cdots\omega_n^{m_n}),$$
where $\mathbf 1=(1,1,\dots, 1)$. The following statements  are equivalent:
\begin{itemize}
\item[$i)$] 
In Roumieu case: $ \exists A>0, \exists c>0,$

in Beurling case: $ \forall A>0, \exists c>0,$
 $$ |D^p F(\omega)|\leq c A^{|p|} M_p,\  p\in\mathbb N_0^{n}, \ \omega\in \mathbf D.$$
\item[$ii)$] 
In Roumieu case: $ \exists h>0, \exists C>0,$

in Beurling case: $ \forall h>0, \exists C>0,$
$$ |a_m|\leq C e^{-M(h|m|)}, \; m\in\mathbb N_0^{n}.$$
\end{itemize}
\end{prop}
\begin{proof}
We will prove the Roumieu case not repeating the arguments for the Beurling case but in the end of the proof  of part one we indicate the changes in the proof of the Beurling case.

 Let us assume that $(i)$ holds. Let $F(\omega)=\sum_{m\in\mathbb N^n} b_m\omega^m,\;\omega\in\mathbf D,$ be the power series expansion of $F$.  Then, 
$$D^p F(\omega)=\sum_{m\geq p}\frac{m!}{(m-p)!} b_m \omega^{m-p},\;\omega\in\mathbf  D$$
and due to Cauchy formula, there exist $C>0, h>1 $ such that
$$\frac{m!}{(m-p)!}|b_m|\leq C h^{|p|} M_p,\; p, m\in\mathbb N_0^n.$$
(Here is the difference with the Beurling case.)

We will use  relations between $a_n$ and $b_n, \; n\in \mathbb N,$ given in \cite{D1} and \cite{JPP}. There holds
\begin{equation*}\label{djp} b_m=\sum_{j\leq m,j\leq \mathbf 1}(-1)^ja_{m-j},\;
 \sum_{k\leq s}b_{m+\mathbf 1+k}= \sum_{k\leq s}\sum_{j\leq\mathbf 1}(-1)^ja_{m+k+\mathbf 1-j}.
 \end{equation*}

An elementary combinatorial result given in \cite{JPP}
implies
 \begin{equation}\label{1djp}
 \sum_{k\leq s}b_{m+\mathbf 1+k}= \sum_{j\leq\mathbf 1}(-1)^{|j|}a_{m+s+\mathbf 1-s^j},
 \end{equation}
with the notation 
$$s_l^j=0 \mbox{ if } s_l=0,\;  \;s_l^j=s_l+1 \mbox{ if }  s_l=1, \; l=1,...,n:$$
$
$
By (\ref{1djp}) one has
\begin{equation}\label{aa}
(-1)^{|n|}a_m=\sum_{k\leq s}b_{m+\mathbf 1+k}-\sum_{\substack{j\leq\mathbf{1}\\ j\not=\mathbf 1}}(-1)^{|j|}a_{m+s+\mathbf 1-s^j}.\end{equation}

Let $m_i\geq m_0, i=1,...,n,$ where $m_0$ is from $(M.00).$ Since $\frac{m!}{(m-p)!}\geq e^{-|p|} m^p$, when  $m\geq p$ (which means $m_i\geq p_i, i=1,...,n$), $ p, m
\in\mathbb N_0^n$, we obtain
\begin{equation}\label{nej11}
|b_m|\leq C \prod_{j=1}^{n} \inf_{p_j\leq m_j} \frac{ (eh)^{p_j} M_{p_j}}{m_j^{p_j}}, m_i, p_i\in\mathbb N_0.
\end{equation}

Now, by $(M.00)$, for every $i\in\{1,...,n\}\quad \mbox{wherever}\  p_i\geq  m_i>m_0$ and $eh>1$,
$$
\frac{ (eh)^{p_i} M_{p_i}} {m_i^{p_i}}\geq \frac { (eh)^{m_i} M_{m_i}} {m_i^{m_i}}.
$$
Thus, by (\ref{nej11}),
\begin{equation}\label{nej22}
|b_m|\leq C \prod_{j=1}^{n} \inf_{p_j\in \mathbb N_0} \frac{ (eh)^{p_j} M_{p_j}}{m_j^{p_j}}, m_i\geq m_0, m_i, p_i\in\mathbb N_0,
\end{equation}
which gives, provided $h_1=1/(eh)$, 
$$|b_m|
\leq \prod_{j=1}^n\inf_{p_j\in\mathbb N_0}
\frac{M_{p_j}}{(\frac{m_j}{eh}) ^{p_j}  }=\prod_{j=1}^ne^{-M(h_1m_j)}, m\in\mathbb N_0^n, m\geq m_0\cdot{\bf 1}=(m_0,m_0,\dots, m_0).$$
In latter we supposed  $m_j\geq m_0, j\in \{1,2,\dots, n\}$.  If it is not the case, namely when $m_{j_0}<m_0$ for some $j_0\in \{1,2,\dots, n\}$, 
the quantity
$$ 
 \inf_{p_{j_0}\leq m_{j_0}} \frac{ (eh)^{p_{j_0}} M_{p_{j_0}}}{m_{j_0}^{p_{j_0}}}\cdot e^{M(h_1 m_{j_0})}
  $$
is finite, i.e. equation (\ref{nej22}) holds for every $m\in \NN^n_0$ with possibly larger constant $C$.
 Since $$\max_{k\in\{1,\dots, n\}}m_k\geq \frac{|m|}{n}\Rightarrow e^{-M(h_1 \max_{k\in\{1,\dots, n\}}m_k)}\leq e^{-M(\frac{h_1}{n} |m|)}, \;j=1,2,\dots, n,$$
 we infer
 $$e^{-M(h_1m_1)}e^{-M(h_1m_2)}\cdots e^{-M(h_1m_{n}) }  \leq e^{-M(h_1\max_{k\in \{1,2,\dots, n\} }m_k)}\leq  e^{-M(\frac{h_1}{n}|m|)}. $$

 Take $h_2={h_1}/{|n|}$. Let $p,m\in \Nd$ such that $p\geq m.$ Then,
 $$M(h_2|p|)\geq \frac{M(h_2|p-m|)+M(h_2|m|)}{2}.$$
 By (\ref{nejednakostM2}), we have ${M(t)}/{2}\geq M({t}/{2^a} )-C, t>0.$ So, we continue,
 $$M(h_2|p|)\geq M(h_3|p-m|)+M(h_3|m|)-C,\; \mbox{ where }\; h_3={h_2}/{2^a}.$$

By (\ref{aa}) one has
 $$|a_m|\leq \sum_{k\leq s}|b_{m+\mathbf{1}+k}|+\sum_{\substack{j\leq\mathbf{1}\\ j\not=\mathbf 1}}|a_{m+s+\mathbf{1}-s^{(j)}}|$$
 \begin{equation} \label{dodato}\leq C_2 e^{-M(h_3|m|) }\sum_{k\in\Nd} e^{-M(h_3|k+{\bf 1}|)}+\sum_{\substack{j\leq\mathbf{1}\\ j\not=\mathbf 1}}|a_{m+s+\mathbf{1}-s^{(j)}}|, \;m\in\Nd.
\end{equation}
Next, $\sum_{k\in\Nd} e^{-M(h_3|k+{\bf 1}|)}<\infty$ and the second sum in (\ref{dodato}) has $2^n-1$ terms. 

Since $a_l\to 0$ when $|l|\to \infty, $ find $k\in \NN_0^n$ such that \\$\ds \sum_{\substack{j\leq\mathbf{1}\\ j\not=\mathbf 1}}|a_{m+s+\mathbf{1}-s^{(j)}}|\leq C_3 e^{-h_3|m|}$ for a  fixed $m$.

So,  by (\ref{1djp}), it follows that
 there exists a new constant $C>0$ such that $$|a_m|\leq C e^{-M(h_3|m|)},\; m\in\mathbb N_0^n.$$
 
Let us note that in the Beurling case one has to reverse this  procedure  where $h_1=1/(eh)$,  $h_2=h_1/|n|$ and $h_3=h_2/2^a$. Actually, for given $h_3>0$ one should find 
$h$ in the previous procedure: $h_2=h_32^a,$ $h_1=h_2|n|$ and $h=1/(eh_1)$ provided $\displaystyle he=\frac{1}{h_1}>1$. With such determined $h>0$ which depend only on $h_3$ one have that for every $h_3>0$ there is a constant
$C>0$ such that $a_m\leq Ce^{h_3|m|}, m\in\mathbb N^n_0.$   So the same proof works for the Beurling case, as well.

Suppose now that $(ii)$ holds. Let $m\in\Nd.$  By  (\ref{nejednakostM2}),
  $$M(h|m|)\geq\frac{ M(hm_1)+M(hm_2)+\cdots M(hm_{n})}{n+2}\geq M(h'm_1)+\cdots M(h'm_{n})-C,$$ for some $C>0$ and $h'={h}/{{n}^a}$.
Now we have with new constants $C>0$ and $h>0,$
$$|a_m|\leq C e^{-M(hm_1)}\cdots e^{-M(hm_{n})},\; m\in\Nd.$$
This implies,  
$$\displaystyle |b_m|\leq\sum_{k\leq m,k\not= \bf{1}}|a_m|\leq {card}\{k\in\Nd, k\leq m, k\not={\bf 1}\}  \cdot  C e^{-M(hm_1)}\cdots e^{-M(hm_{n})} $$
$$\leq (m_1+1)\dots (m_n+1)\cdot C e^{-M(hm_1)}\cdots e^{-M(hm_{n})}\leq  C_1 e^{-M(h_1m_1)}\cdots e^{-M(h_1m_{n})} $$

for some other constants $C_1, h_1>0$.
Further on, for every $\omega\in \bf D$:
$$|D^p F(\omega)|\leq \sum_{m\geq p}\frac{m!}{(m-p)!}|b_m|\leq C_1\sum_{m\in\Nd} \prod_{j=1}^{n} m_j^{p_j} e^{-M(h_1 m_j)}$$
$$\leq C_2 \sum_{m\in\Nd} \prod_{j=1}^{n} m_j^{p_j} e^{-2M(h_2m_j)}, m\in \Nd,$$
where we used (\ref{nejednakostM2}) with suitable new constants $C_2$ and $h_2$. It remains to utilize the obvious inequality
$$\prod_{j=1}^{n}e^{-M(h_2 m_j)}\leq e^{-M(h_2{|m|}/{n})},\; n\in \Nd,$$
as well as the following fact
$$\prod_{j=1}^{n} m_j^{p_j }e^{-M(h_2 m_j)}\leq \prod_{j=1}^{n} \frac{1}{h_2^{p_j}} M_{p_j}= \Big(\frac{1}{h_2}\Big)^{|p|}M_{|p|},\; p\in\mathbb N^n.$$
Since $\sum_{m\in \Nd}e^{-M(h_2\frac{|m|}{n})}<\infty $, the proof is completed.
\end{proof}

The next step towards the proof is the following assertion which traced back to  \cite[Theorem $(3.3)$]{D1}, in the case $n=1$.
\begin{prop}\label{3.3}
Let $f\in L^2(\mathbb R^{n}_+)$ such that for some $A>0$ and some $C>0$, respectively, for every $A>0$ exists $C>0$ (this is the Beurling case), 
$$|t^k f(t)|\leq C A^{|k|} M_{k}, k\in\Nd, \;t\in\mathbb R^{n}.$$ Then
for every $h>0$ exists $C>0$, respectively, for some $h>0$ exists $C>0$, such that
\begin{equation}\label{fd} |(\mathcal F_D(f))^{(p)}(w)|\leq C h^{|p|}  M_{|p|} \ \mbox{for every} \ p\in \Nd,\  w\in \mathbf D_{\emptyset},
\end{equation}
where
$$\mathbf D_{\emptyset}=\{w=(w_1,\dots w_n)\in \mathbf D^{n} : \  {\rm Re}(w_i)\leq 0,i=1,2,\dots, n\}$$
\end{prop}
\begin{rem}
This proposition implies that for $f\in L^2(\mathbb R^{n}_+)\cap C^\infty(\mathbb R^n_+)$, for every $s\in\mathbb N_0^n$,
$$|t^k f^{(s)}(t)|\leq C A^{|k|} M_{k}, k\in\Nd, \;t\in\mathbb R^{n}$$
implies (\ref{fd}) with the corresponding constants in the Roumieu and in the Beurling case, respectively.
\end{rem}
\begin{proof}
The proof is a modification of the proof of \cite[Theorem $3.3$]{D1}. We prove the assertion only in the Roumieu case. Beurling case is similar. 

 Let $p\in\mathbb N^n.$ Let $w\in D_{\emptyset}$.  From the initial assumption and \cite[(3.7) and (3.8)]{D1} we have
$$\Big|\int_{\mathbb R^{n}} (-t)^{k} f(t) \cdot \prod_{j=1}^{n} e^{-\frac{1}{2} \frac{1+w_j}{1-w_j} t_j} dt\Big|\leq C_1 A_1^{|k|} M_{k},
\ k\in\Nd.$$
So, we conclude that for $p\in\mathbb N_0^n,$
$$ |(\mathcal F_D(f))^{(p)}(w)|\leq$$
$$
\sum_{k\leq p}\frac{1}{k!} \sum_{m\leq k}{k\choose m}\Big( \prod_{j=1}^n m_j (m_j+1)\cdots (m_j+p_j-1)\Big) C_1 A_1^{|k|} M_{k}.$$
Since $$
\sum_{m\leq k} {k\choose m}\leq 2^{|k|}, \;m, k\in \Nd,\ \  \frac{|k|!}{k!}\leq n^{|k|},\  \mbox{and}
$$ 
  $$\prod_{j=1}^n m_j (m_j+1)\cdots (m_j+p_j-1)\leq  |p|(|p|+1)\cdots (2|p|)\leq 
   \frac{(2|p|)!}{(|p|-1)!},$$it follows
$$
|(\mathcal F_D(f))^{(p)}(w)| \leq  C_1  A_2^{|p|} \frac{(2|p|)!}{(|p|-1)!} \sum_{k\leq p}\frac{1}{|k|!}  M_{k},\; \omega\in \mathbf D_{\emptyset},
 $$
where $A_2=2nA_1$. Further on,   we conclude 
$$\frac{(2|p|)!}{(|p|-1)! \cdot |k|!\cdot (|p-k|+1)!}\leq  3^{2|p|},\; k, p, \Nd, k\leq p\Rightarrow$$
$$ \frac{(2|p|)!}{(|p|-1)! \cdot |k|!}\leq 9^{|p|} (|p-k|+1)!\leq A_3^p |p-k|! $$

 for suitable $A_3>0.$ By condition  $(M.0^s),$ 
 
 $$ \frac{(2|p|)!}{(|p|-1)! \cdot |k|!}\leq  A_3^p |p-k|! \leq A_4^{|p|} M_{p-k}$$ 
 
provided suitable $A_4>0$,
 and so we continue, utilizing $(M.2)$:
 
 $$
|(\mathcal F_D(f))^{(p)}(w)| \leq C_1(A_2A_4)^{|p|} \sum_{k\leq p}M_{p-k}  M_{k}
 $$
$$ \leq C_1 A_5^{|p|} M_p \sum_{k\leq p}1\leq C_1 A_5^{|p|} M_p\cdot p_1p_2\cdots p_n$$
$$\leq C_1 A_5^{|p|} M_p |p|^n\leq C_2 A_6^{|p|} M_{p} ,\; p\in \Nd,$$
where  $A_5=HA_2A_4$  ($H$ is the constant from $(M.2)$)  and $C_2,A_6$ are suitable constants for which $C_1 A_5^{l}l^n \leq C_2 A_6^l$ holds for every $l\in\mathbb N_0.$ This concludes the proof. 
\end{proof}
For the last assertion we follow \cite{D1} (case $n=1$) and \cite{JPP} (case $n>1$).
\begin{prop} Let $f\in\mathcal G^*(\mathbb R^n_+)$. Then, in the Roumieu case, there exist constant $C,A>0$, in Beurling case, for every $A>0$ exists $C>0$, such that
\begin{equation}\label{exp2}
|D^p\mathcal F_D(f)(\omega)|\leq C A^{|p|} M_{p},\; \omega\in \mathbf D,  p\in \Nd .\end{equation}
\end{prop}
\begin{proof}
We   prove the assertion only in the Roumieu case.
 Let $f\in\mathcal G^{\{M_p\}}(\mathbb R^n_+)$ and $f=\sum_{m\in\mathbb N^n_0}a_m\ell_m$. By (\ref{exp1}), $$ \mathcal H_0^{\Lambda'}f= \sum_{m\in\mathbb N^n}(-1)^{n_1^{\lambda'}+...+n_1^{\lambda'}}a_m\ell_m.$$
 With the notation (as in \cite{JPP}), for $\omega\in \mathbf D_{(\Lambda')}, $ where $\Lambda'=\{1,2,\dots n'\}$, $\Lambda''=\{1,2,\dots, n''\}$ and $\Lambda=(\Lambda',\Lambda'')$,
$$\widetilde \Lambda'\omega=\zeta, \mbox{ where }\zeta_{\lambda'_l}=-\omega_{\lambda'_l},\; l=1,...,d', \; \zeta_{\lambda''_s}=w_{\lambda''_s},\; s=1,...,d''$$ gives $\Lambda'(\mathbf D_{(\Lambda')})= \mathbf D_{\emptyset}$ and 
there holds  
$$\mathcal F_{\mathbf D}(f)(w)=\prod_{l=1}^{d'}\frac{1+w_{\lambda_l'}}{1-w_{\lambda_{l'}}}\mathcal F_{\mathbf D}(\mathcal H^{\Lambda'}_{0}f)(\widetilde \Lambda'w),\; w\in D. 
$$

Since $\mathcal H_0^{(\Lambda')}f\in\mathcal G^{\{M_p\}}$, from Theorem \ref{3.3} (and remark after this theorem) it follows
\begin{equation}\label{5.15}
|D^m \mathcal F_D(\mathcal H_0^{(\Lambda')} f)(\omega)|\leq C_1 A_1^{|m|}M_m,  \;m\in\Nd, w\in \mathbf D_{\emptyset},.\end{equation}
And now we  have, utilizing Leibniz rule and the estimates from \cite[(5.16)]{JPP} for $m\in\Nd$, 
$$|D^m \mathcal F_D f(w)|\leq \sum_{r_{\Lambda'}\leq m_{\Lambda'}} {m_{\Lambda'}\choose r_{\Lambda'}}2^{d'} r_{\Lambda'}!|D^{m_{\Lambda'}-r_{\Lambda'}} D^{m_{\Lambda''}}_{w_{\Lambda''}}\mathcal F_{\mathbf D}( \mathcal H^{(\Lambda')}f)(\tilde \Lambda' w)|$$
$$\leq C_2  \sum_{r_{\Lambda'}\leq m_{\Lambda'}} {m_{\Lambda'} \choose  r_{\Lambda'}} M_{r_{\Lambda'}} A_2^{|m-r_{\Lambda'}|}M_{m_{\Lambda'}-r_{\Lambda'}}M_{m_{\Lambda''}}\leq C_3 A_3^{|m|}M_m.$$
In the latter we named $C_2=C_1 2^{d'}$. We have also used  that $$ M_{r_{\Lambda'}}M_{m_{\Lambda'}-r_{\Lambda'}}M_{m_{\Lambda''}} \leq M_{m}$$ which  follows from $(M.1)$ and since $$ \sum_{r_{\Lambda'}\leq m_{\Lambda'}} {m_{\Lambda'}\choose r_{\Lambda'}}\leq 2^{|m_{\Lambda'}|},$$ we  take $A_3=2A_2.$ The proof is completed.
\end{proof}

Again, with the quoted assumption on $\{M_p\}$ the associated function of the form (\ref{assf}) and the orthonormal basis $\ell_i,  i\in \mathbb N_0$ of $L^2(\mathbb R^n_+)$, 
 due to the previous proposition, (\ref{5.1}) and  (\ref{5.4}), we have:
\begin{thm}
Let $f\in L^2(\mathbb R^{n}_+)$ and $a_m=\int_{\mathbb R^{n}_+} f(t) \ell_m(t) dt,m\in\Nd$. The following are equivalent:
\begin{itemize}
\item[$i)$] $\exists a>0$, respectively, $\forall a>0$, $\exists C=C_a>0$,  $$|a_m|\leq C e^{-M(a|m|)},\; m\in \Nd,$$
\item[$ii)$] $f\in\mathcal G^{\{M_p\}}(\mathbb R^n_+),$ respectively $f\in\mathcal G^{(M_p)}(\mathbb R^n_+).$
\end{itemize}
\end{thm}

\section{Structural theorems}\label{Sec4}

In this section we show that the expansions through orthonormal basis of all involved spaces do not depend on the order of summations. We show this in the  case
$\mathcal E^*(\mathbb S^{n-1})\widehat{\otimes}\mathcal S(\mathbb R^m)$; the case of the tensor product of $k+2$ test spaces can be similarly examined.

\begin{prop}\label{karakterizacija}
Space $\mathcal E^*(\mathbb S^{n-1})\widehat{\otimes}\mathcal S^*(\mathbb R^{m})$ consists  of  functions with the following expansion,
\begin{multline}\label{br} a(\theta,x)=\sum_{i=0}^{\infty}\Big( \sum_{j=0}^{\infty}\sum_{k=0}^{N_j} a_{k,j,i}Y_{k,j}(\theta) \Big)u_i(x)=\\ =\sum_{j=0}^{\infty}\sum_{k=0}^{N_j}\Big(\sum_{i=0}^{\infty} a_{k,j,i}u_i(x)\Big)Y_{k,j}(\theta),
\theta \in\mathbb S^{n-1}, \;x\in\mathbb R^m,\end{multline}
where the expressions in brackets converge uniformly with respect to  convergence in $\mathcal E^*(\mathbb S^{n-1})$ and $\mathcal S^*(\mathbb R^m)$. There holds ($\theta\in\mathbb S^{n-1}, \;x\in\mathbb R^m$),
\begin{equation*}\label{11L}\exists h>0, \exists l>0, \mbox{ respectively } \forall h>0, \forall l>0, \exists C=C_{h,l}>0,
\end{equation*}
$$|a_{k,j,i}|\leq  C e^{-M(h\mu_i^{{1}/{q}})-M(lj)},\;  i\in\mathbb N_0, j\in\mathbb N_0, k\leq N_j.
$$
Next,
 $$u(\theta,x)=\sum_{i=0}^{\infty} \sum_{j=0}^{\infty}\sum_{k=0}^{N_j} b_{k,j,i}Y_{k,j}(\theta) u_i(x)\in\mathcal E'^*(\mathbb S^{n-1})\widehat{\otimes}\mathcal S'^*(\mathbb R^{m}),$$
if and only if
\begin{equation}\label{111L}\forall h>0, \forall l>0, \mbox{ respectively } \exists h>0, \exists l>0, \exists C=C_{h,l}>0,
\end{equation}
$$|b_{k,j,i}|\leq  C e^{M(h\mu_i^{{1}/{q}})+M(lj)},\;   i\in\mathbb N_0, j\in\mathbb N_0, k\leq N_j.
$$
\end{prop}
 Note that by \cite{GPR} it follows that with suitable $C>0$,

$$
 \mu_i^{1/q}\sim  Ci^{1/2m}, i\rightarrow \infty.
$$
\begin{proof}
We give the proof for test functions in the Roumieu case. The Beurling case is similar.

Fourier coefficients of $a\in \mathcal E^{\{M_p\}}(\mathbb S^{n-1})\widehat{\otimes}\mathcal S^{\{M_p\}}(\mathbb R^{m})$ are given by
 $$a_{k,j,i}=\int_{\mathbb S^{n-1}\times \mathbb R^{m}}  a(\theta, x)  u_i(x)Y_{k,j}(\theta) dxd\theta,\; i\in\mathbb N_0, j\in\mathbb N_0, k\leq N_j.$$
Recall that $a$ satisfies
$$
\Big|x^p \partial_x^{\alpha}\big( \partial_{\theta}^{\beta}a(\theta,x)\big)\Big|\leq C_1h_1^{|\alpha+\beta+p|}M_{|\alpha+p|}M_{|\beta|},
$$
\begin{eqnarray}\label{jednacina1.5}\theta\in\mathbb S^{n-1}, x\in\mathbb R^m, p,\alpha\in\mathbb N_0^{m}, \beta\in\mathbb N_0^{n-1}.
\end{eqnarray}
Set
$a(\theta,x)=\sum_{i=1}^{\infty} a_i(\theta)u_i(x), \theta\in\mathbb S^{n-1}, x\in\mathbb R^m,$
 where $$a_i(\theta)=\int_{\mathbb R^n} a(\theta,x) u_i(x)dx, \theta\in\mathbb S^{n-1},\;  i\in\mathbb N.$$
Let $\beta=0$ in (\ref{jednacina1.5}).  
  By the careful examination of \cite[Theorem $4.1$]{VV2} and (\ref{vazno1}), we  conclude that   for some $h_2$ and $C_2>0$,$$|a_i(\theta)|\leq C_2 e^{-M(h_2 \mu_i^{{1}/{q}})}, \; i\in\mathbb N, $$ uniformly for $\theta\in{\mathbb S}^{n-1}$  ($q$ is order of $P$). What is more, since $$\partial_{\theta}^{\beta}a(x,\theta)=\sum_{i=0}^{\infty} \partial_{\theta}^{\beta} a_i(\theta)u_i(x),$$ the same argument leads to the conclusion that for some $h_3, h_4, C_3>0$, $$|\partial_{\theta}^{\beta} a_i(\theta)|\leq C_3 h_3^{|\beta|}  M_{|\beta|}e^{-M(h_4 \mu_i^{{1}/{q}})}, \;\theta\in\mathbb S^{n-1}, \beta\in\mathbb N_0^{n-1}.$$
This basically proves the ultradifferentiability of the function $a_i$ on the unit sphere, for every $i\in\mathbb N_0.$ It remains to utilize \cite[Theorem $4.1$]{VV1}
  in order to prove that coefficients of
$$a_i(\theta)=\sum_{j=0}^{\infty}\sum_{k=0}^{N_j} a_{k,j,i} Y_{k,j}(\theta),\; \theta\in\mathbb S^{n-1},$$
 satisfy:  $\forall i\in\mathbb N_0,\exists h_5,C_5>0,$
$$|a_{k,j,i}|\leq C_5  e^{-M(h_5\mu_i^{{1}/{q}})-M(lj)},\; j\in\mathbb N_0,  \mbox{ for some } h_5>0, C_5>0.$$
 Note that, for $i\in\mathbb N_0, j\in\mathbb N_0, k\leq N_j,
 $
 $$
 a_{k,j,i}=\int_{\mathbb S^{n-1}} a_i(\theta) Y_{k,j}(\theta) d\theta=\int_{\mathbb S^{n-1}} \Big(\int_{\mathbb R^{m}} a(\theta, x) f_i(x)dx\Big)Y_{k,j}(\theta)d\theta. $$
 This shows the uniform convergence on the term in bracket on the left side of  (\ref{br}). A different approach, differentiating  the sum
$$a(\theta,x)=\sum_{j=0}^{\infty}\sum_{k=0}^{N_j} a_{k,j}(x)Y_{k,j}(\theta),\theta\in\mathbb S^{n-1}, \; x\in\mathbb R^m,$$
first in $x$ and than in $\theta$ in (\ref{jednacina1.5}), where $$a_{k,j}(x)=\int_{\mathbb S^{n-1}} a(\theta, x) Y_{k,j}(\theta)d\theta, j\in\mathbb N_0, k\leq j,$$
leads to the same conclusion for the right hand side of (\ref{br}). This completes the proof of (\ref{br}).

Using the first part of this proposition, we are able to characterize the coefficients in the expansion of an arbitrary element of

\noindent $ \mathcal E'^{\{M_p\}}(\mathbb S^{n-1})\widehat{\otimes}\mathcal S'^{\{M_p\}}(\mathbb R^m)$, by (\ref{111L}). We  skip this part.
\end{proof}

\subsection{Renumbering of eigenvalues of  $E=E_y$}\label{Sec4.1}
Now, we  rearrange the sequence of eigenvalues for $E$.
 Eigenvalues  $|s|, s\in\mathbb N_0^n$
are re-indexed as follows.
Let $e=(e_1,...,e_n)\in\mathbb N_0^n.$ We start with $j_0=0$, which corresponds to $|e|=0$. The number of points $e\in\mathbb N^n$ so that $|e|= t$ is equal to
${{n+t-1}\choose{t-1}}.$
Thus, the first coordinate with $|e|= t+1$ takes the position
$$j=1+n+{{n+1}\choose{2}}+...{{n+t-1}\choose{t-1}}=1+{{n+t}\choose{t}},\; t\in\mathbb N,$$
that is,
we obtain that the coefficient with index $|n|=t+1$ correspond to $$j\in\Bigl\{1+{{n+t}\choose{t}},2+{{n+t}\choose{t}},...,{{n+t+1}\choose{t+1}}\Bigr\}$$

\noindent There exist $C>0$ and $c>0$ such that
$c t^n \leq {{n+t}\choose{t}}\leq C t^n.
$
we infer (with new suitable $c>0$) that
$$E\ell_p=\nu_p\ell_p, p\in\mathbb N_0, \mbox{ where } \nu_p\sim c p^{1/n} \mbox{ as } p\rightarrow \infty.
 $$
 Thus, we reformulate the following
necessary and sufficient condition:
$$\phi=\sum_{p=0}^\infty a_p\ell_p\in \mathcal G^*(\mathbb R^n_+)
\mbox{ if and only if }$$
$$ \;\forall r>0, resp.,\exists r>0, \exists C=C_r>0,\;\;
 |a_p|\leq Ce^{-M(rp^{1/n})},\;p\in\mathbb N_0.
$$
Also,
$$f=\sum_{p=0}^\infty b_p\ell_p\in \mathcal G'^*(\mathbb R^n_+)
\mbox{ if and only if }$$
$$ \;\exists r>0, resp.,\forall r>0, \exists C=C_r>0,\;\;
 |b_p|\leq Ce^{M(rp^{1/n})},\; p\in\mathbb N_0.
$$

Now we give the structural theorem in the more general case.
We first prepare notation as follows. From now on, eigenvalues of $A_j $ will be denoted by $\lambda_{j}^s,\; s=1,...,k,\; j\in\mathbb N_0.$


\begin{prop}\label{gen-karakterizacija}
Space  $\mathcal H^*(\prod_{s=1}^kX_s\times\mathbb R^n_+\times\mathbb R^m)$ consists of smooth functions of the form
$$\phi(t_1,...,t_k,y,x)=\sum_{j^1=0,...,j^{k}=0,p=0,i=0}^{\infty,...,\infty}a_{j^1,...,j^k,p,i}v_{j^1}(t_1)...v_{j^k}(t_k)\ell_p(y)u_i(x),
$$
where $t_s\in X_s, s=1,...,k,$ $y\in\mathbb R^n_+$, $x\in\mathbb R^m$ and the sums are equal for any order of  summation. There holds,
\begin{equation*}\label{gen11L}\exists h>0, \exists r>0, \exists l^s>0,\; s=1,...,k,\end{equation*}
$$\mbox{ respectively, } \forall h>0, \forall r>0, \forall l^{s}>0, \;s=1,...,k, \exists C=C_{h,r}>0,
$$
$$|a_{j^1,...,j^k,p,i}|\leq  C e^{-\big(\sum_{s=1}^k M(l^s{j^s})+M(rp^{1/n})+M(h\mu_i^{1/q})\big)}.$$

Next, $U\in \mathcal H'^*(\prod_{s=1}^kX_{s})\times\mathbb R^n_+\times\mathbb R^m)$,  of the form
\begin{multline*}U(t_1,...,t_k,y,x)=\sum_{j^1=0,...,j^{k}=0,p=0,i=0}^{\infty,...,\infty}b_{j^1,...,j^k,p,i}v_{j^1}(t_1)...v_{j^k}(t_k)\ell_p(y)u_i(x),                                     \\ t_s\in X_s, s=1,\dots, k,y\in\mathbb R_{+}^n,\; x\in\mathbb R^m,
\end{multline*}
if and only  if
\begin{equation*}\label{gengen11L}\forall h>0, \forall r>0, \forall l^s>0,\; s=1,...,k,\end{equation*}
$$ \mbox{ respectively, } \exists h>0, \exists r>0, \exists l^{s}>0,\; s=1,...,k, \exists C=C_{h,r}>0,
$$
$$|b_{j^1,...,j^k,p,i}|\leq  C e^{\sum_{s=1}^k M(l^sj^s)+M(rp^{1/n})+M(h\mu_i^{1/q})}.$$

\end{prop}
For the proof  (which can be given by  induction) one has to  repeat  arguments used in the proof of  Proposition \ref{karakterizacija}. We skip this part.

 \section{Equations}\label{Sec5}


\subsection{Equation $L_{3} u=f $} Our starting point is the  equation (\ref{123}).
 Let    $f\in \mathcal E'^*(\mathbb S^{n_1-1}\times\mathbb R_+^{n_2}\times\mathbb R^{n_3}),\; f\neq 0,$ have the form
 \begin{multline}\label{formf3}\displaystyle f(\theta,x,y)=\displaystyle \sum_{j\in\mathbb N_0, k\leq N_j, p\in \mathbb N_0, i\in\mathbb N_0} b_{k,j,p,i} Y_{k,j}(\theta) \ell_p(y)u_i(x),\\
 \theta\in\mathbb S^{n_1-1},\ y\in\mathbb R_+^{n_2},\ x\in \mathbb R^{n_3}.
 \end{multline}

We assume for (\ref{123}) that $c_1, c_2, c_3$ are different from zero and that $c_1=1$ so that it becomes
\begin{equation}\label{tri}  
\Delta_{\mathbb S^{n-1}}^{h}+c_2E^d+c_3Pu=f
\end{equation} 
 Other possible cases can be treated in a similar way. The solution  is expected in the form
\begin{equation}\label{formu3}
u(\theta,y,x)=\sum_{i=1}^{\infty}\sum_{j=0}^{\infty}\sum_{k=0}^{N_j} a_{k,j,p,i} Y_{k,j}(\theta)\ell_p(y)u_i(x), \theta\in\mathbb S^{n-1}, y\in\mathbb R^{n}_+, x\in\mathbb R^{m}.
\end{equation}

Simple  calculation  gives
 $$
 \displaystyle \sum_{j\in\mathbb N_0, k\leq N_j, p\in \mathbb N_0, i\in\mathbb N_0} a_{k,j,p,i}(j^h(j+n_1-2)^h+ c_2\nu_p^{d}+c_3\mu_i) Y_{k,j}(\theta) \ell_p(y)u_i(x)
 $$
 $$=
 \displaystyle \sum_{j\in\mathbb N_0, k\leq N_j, p\in \mathbb N_0, i\in\mathbb N_0} b_{k,j,p,i}Y_{k,j}(\theta) \ell_p(y)u_i(x).$$
 (Recall, $\nu_p$ are eigenvalues of $E_y$ and $\mu_i$ are eigenvalues of $P$). Formally,
 $$ a_{k,j,p,i}=\frac{b_{k,j,p,i}}{j^h(j+n_1-2)^h+ c_2\nu_p^{d}+c_3\mu_i}.$$
 Equations
\begin{equation}\label{3us} j^h(j+n_1-2)^h+ c_2\nu_p^{d}+c_3\mu_i=0, \; i\in\mathbb N_0, p\in\mathbb N_0, j\in\mathbb N_0, k\leq N_j,
\end{equation}
have the main role for the solvability.
\
Let
$$F_3=\{(i,p,j)\in\mathbb N_0^{3}: (\ref{3us}) \mbox{ holds}\},
$$
$$  E_{3}=\{(i,p,j)\in\mathbb N_0^{3}: b_{k,j,p,i}= 0\}.$$
We assume in the sequel that $k\leq N_j$ is included in the definitions of index sets $F_3$ and $E_3.$
\begin{prop}
Condition $  F_3\subset E_3$ is the necessary one for the solvability of  (\ref{tri}).
\end{prop}
If $F_3$ is empty, then
\begin{equation}\label{ab3} a_{k,j,p,i}=\frac{b_{k,j,p,i}}{j^{h}(j+n-2)^{h}+ c_2\nu_p^{d}+c\mu_i}\in\mathbb C,
\end{equation}
$$ i\in\mathbb N_0,p\in\mathbb N_0, j\in\mathbb N_0,\; k\leq N_j.
$$

\begin{rem} We will not treat the case when 
$$F_3\neq\emptyset, F_3\subset E_3 \mbox{ but } F_3\neq E_3.$$
The analysis which is to follow shows how one can analyse this case.
\end{rem}

For the next proposition we need:

 \begin{defn}\label{5.2}
A couple of real numbers $(c_2, c_3)$  is not an $R-$ respectively,  $B-$Liouville couple  with respect to the  sequences
$\nu_p^d$ and $\mu_i$  of eigenvalues of $E^d$ and $P(x,D)$
 iff for every $\varepsilon>0$, respectively, some $\varepsilon>0,$ there exist $C=C_{\varepsilon}>0,$ such that
\begin{equation}\label{liuvil1}
\inf_{\tau\in\mathbb Z} |\tau + c_2\nu_p^{d}+c_3\mu_i|\geq C_{\varepsilon} e^{-M(\varepsilon p)}e^{-M(\varepsilon \mu_i^{{1}/{q}})},
 p\in\mathbb N_0, i\in\mathbb N.
\end{equation}
\end{defn}
\begin{prop}\label{thm123}
Let $f\in \mathcal {H}'^*(\mathbb S^{n_1-1}\times\mathbb R^{n_2}_+\times\mathbb R^{n_3})$  be of the form (\ref{formf3}).

(i) Assume that $F_3=\emptyset$ and $E_3=\emptyset$
so that (\ref{ab3}) holds.

Suppose that  $(c_2,c_3)$ is not a $R$-Liouville, respectively, $B$-Liouville couple with respect to the sequences $\{\nu_p^{d}; p\in\mathbb N_0\}$ and $\{\mu_i; i\in\mathbb N_0\}$.  Then, $u$ is the unique solution to (\ref{123}) with the coefficients of  the form (\ref{formu3}).
Moreover,  if $f\in \mathcal {H}^*( \mathbb S^{n_1-1}\times\mathbb R^{n_2}\times\mathbb R^{n_3})$, then $u\in \mathcal {H}^*\mathbb ( \mathbb S^{n_1-1}\times\mathbb R^{n_2}\times\mathbb R^{n_3})$, as well.

(ii) Assume that $E_3=F_3\neq\emptyset.$ and that (\ref{liuvil1}) holds over the complement of the set $E_3.$

a) Assume $card\;E_3<\infty$. Then the solutions exist with \\  $a_{k,j,p,i}, (i,p,j)\in F_3$ arbitrary chosen. Moreover for these solutions, the hypoellipticity holds if $f\in \mathcal E'^*(\mathbb S^{n_1-1}\times\mathbb R_+^{n_2}\times\mathbb R^{n_3}).$ 

b) Assume $card\;E_3=\aleph_0.$ Then for the solvability and hypoellipticity one has to assume conditions of assertion $(ii)$ so that the same conclusions as in (ii) hold.
 
\end{prop}
\begin{proof} (i)
We prove the assertion in the Roumieu case since the proof in the Beurling case is  similar.

 By the use of (\ref{liuvil1}) one has
 that $u$ of the form (\ref{formu3}), satisfies: for every $\varepsilon>0$ and corresponding $C_\varepsilon>0,$
$$|a_{k,j,p,i}| = \frac{|b_{k,j,p,i}|}{j^h(j+n-2)^h+c_2\nu_p^d+c_3\mu_i}\leq|b_{k,j,i}|  e^{M(\varepsilon p)}e^{M(\varepsilon \mu_i^{{1}/{q}})}/C_\varepsilon,$$
$$ i\in\mathbb N_0, p\in\mathbb N_0, j\in\mathbb N_0, k\leq N_j.$$  This is a direct consequence of  the assumptions on $b_{k,j,p,i}$ combined with  the fact that $(c_2,c_3)$ is not an $R-$ Liouville couple. Thus,
   $u\in\mathcal {H}'^{\{M_p\}} (\mathbb S^{n-1}\times\mathbb R^n_+\times\mathbb R^{m}).$ Using the properties of $b_{k,j,p,i}$ for the hypoellipticity, it follows that $u\in\mathcal {H}^{\{M_p\}} (\mathbb S^{n-1}\times\mathbb R^n_+\times\mathbb R^{m})$ if $f$ is in this space, as well.

(iii) a) One can see that  excluding  indices in $E_3$ one can apply the assertion (i). Moreover, arbitrary $a_{k,j,p,i},$ with indices in $E_3,$  can be chosen.

b) The proof is the same as in (ii) b).
\end{proof}

\subsection{General case}\label{Sec5.1}

Now we are considering equation (\ref{jedn}).

Assume that $u\in\mathcal H'^*(\prod_{s=1}^kX_{s}\times\mathbb R^n_+\times\mathbb R^m)$ has the form

$$U(t_1,...,t_k,y,x)=\sum_{j^1=0,...,j^{k}=0,p=0,i=0}^{\infty,...\infty}a_{j^1,...,j^k,p,i}v_{j^1}(t_1)...v_{j^k}(t_k)\ell_p(y)u_i(x),
$$
 $$t_s\in X_s,\; s=1,2,\dots, k, \theta\in\mathbb S^{n-1}, y\in\mathbb R^{n}_+,\; x\in\mathbb R^{m}. $$

Then equation (\ref{jedn}) gives
\begin{multline*}\label{genjedn}
LU=\sum_{j^1=0,...,j^{k}=0,p=0,i=0}^{\infty,...\infty}a_{j^1,...,j^k,p,i}
(a_1(\lambda^1_{j^1})^{h_1}+...+a_k(\lambda^k_{j^k})^{h_k}+b\mu_p^{d}+c\mu_i)
\\ v_{j^1}(t_1)...v_{j^k}(t_k)\ell_p(y)u_i(x),    \\ t_s\in X_s, s=1,2,\dots, k, \theta\in\mathbb S^{n-1}, y\in\mathbb R^{n}_+, x\in\mathbb R^{m}                  .
\end{multline*}
Let  $f\in\mathcal H'^*(\prod_{s=1}^kX_s\times\mathbb R^n_+\times\mathbb R^m)$ have the form
$$f(t_1,...,t_k,y,x)=\sum_{j^1=0,...,j^{k}=0,p=0,i=0}^{\infty,...\infty}b_{j^1,...,j^k,p,i}v_{j^1}(t_1)...v_{j^k}(t_k)\ell_p(y)u_i(x).
$$
 We consider those indices for which 
 \begin{equation}\label{nus}
 a_1(\lambda^1_{j^1})^{h_1}+...+a_k(\lambda^k_{j^k})^{h_k}+b\nu_p^{d}+c\mu_i=0,
 \end{equation}
 holds and indices for which
 $b_{j^1,...,j^k,p,i}=0$.
  Let 
$$F_{k+2}=\{(j^1,...j^k,p,i)\in\mathbb N_0^{k+2}: (\ref{nus}) \mbox{ holds}\},
$$
$$ E_{k+2}=\{(j^1,...,j^k,p,i)\in\mathbb N_0^{k+2}:  b_{j^1,...,j^k,p,i}= 0\}.$$

The coefficients of  a solution $U$ of (\ref{jedn}) should satisfy the infinite  system of equations, for $(j^1,...j^k,p,i)\in\mathbb N_0^{k+2},$
$$a_{j^1,...,j^k,p,i} [a_1(\lambda^1_{j^1})^{h_1}+...+a_k(\lambda^k_{j^k})^{h_k}+b\nu_p^d
+c\mu_i]=b_{j^1,...,j^k,p,i}.
$$
This immediately gives:
\begin{prop}
The necessary condition for the solvability of (\ref{jedn}) is that $F_{k+2}\subset E_{k+2}$.
\end{prop}

\begin{rem} We   repeat the part of the discussion of the previous subsection but we do   not extend Definition \ref{5.2}  since  this is technically more demanding.  Moreover, as in Proposition \ref{123} we will consider the case $\emptyset\neq F_{k+2}=E_{k+2}$ with additional assumptions and  note that   the case when
$F_{k+2}\neq\emptyset, F_{k+2}\subset E_{k+2} \mbox{ but } F_{k+2}\neq E_{k+2},$ can be treated in a similar way.
\end{rem}

We consider the following assumption on $(a_1,...,a_k,b,c)$ (see (\ref{nus})) 
for $(j^1,...,j^k,p,i)\in E_{k+2}:$
\begin{equation}\label{josj}\forall \varepsilon>0, \mbox{ respectively }\exists \varepsilon>0, \exists C=C_\varepsilon>0,
\end{equation}
$$
|a_1(\lambda^1_{j^1})^{h_1}+...+a_k(\lambda^k_{j^k})^{h_k}+b\nu_p^{d}+c\mu_i|\geq C
e^{-\sum_{s=1}^kM(\varepsilon j^s)-M(\varepsilon p^{1/n})-M(\varepsilon\mu_i^{1/q})}.
$$
\begin{cor}Assume $F_{k+2}=E_{k+2}=\emptyset.$ Then eqution (\ref{jedn})
has a unique solution if (\ref{josj}) holds. Moreover, (\ref{jedn}) is hypoelliptic if (\ref{josj}) holds.
\end{cor}
\begin{proof}
 The solution $U$ of (\ref{jedn}) with the coefficients
$$a_{j^1,...,j^k,p,i}=\frac{b_{j^1,...,j^k,p,i}}{a_1(\lambda^1_{j^1})^{h_1}+...+a_k(\lambda^k_{j^k})^{h_k}+b\nu_p^d+c\mu_i},
$$
$$(j^1,...j^k,p,i)\in\mathbb N_0^{k+2},
$$
is the unique one in $\mathcal H'(\prod_{s=1}^kX^s\times\mathbb R^n_+\times\mathbb R^m)$ because of (\ref{josj}). Also, one can easily see that the same condition (\ref{jedn}) implies the hypoellipticity of (\ref{jedn}).
\end{proof} 

Now we consider the case $F_{k+2}=E_{k+2}.$ 
If $F_{k+2}$  is a finite set, then the solvability and hypoellipticity of solutions are clear. The set of solution is infinite since for all indices in $F_{k+2}$ one can choose for $a_{j^1,...,j^k,p,i}$ arbitrary complex numbers.

\begin{cor} Let $f\in\mathcal H'^*(\prod_{s=1}^kX_s\times\mathbb R^n_+\times\mathbb R^m)$,  $F_{k+2}=E_{k+2}=\aleph_0.$ Let 
\begin{itemize}
\item[(g1)]
$$\forall h>0, \forall r>0, \forall \ell^s>0, s=1,...,k, \mbox{ respectively, } $$
$$\exists h>0, \exists r>0, \exists \ell^s>0, s=1,...,k, \exists C=\max\{C_{r,s,\ell^s}>0;  s=1,...,k\},
$$
$$ |a_{j^1,...,j^k,p,i}|\leq C e^{\sum_{s=1}^kM(\ell^s j^s)+M(rp^{1/n})+M(h\mu_i^{1/q})},\;
(j^1,...,j^k,p,i)\in E_{k+2},
$$
\end{itemize} 
respectively,
\begin{itemize}
\item[(g2)]
$$\forall h>0, \forall r>0, \forall \ell^s>0, s=1,...,k, \mbox{ respectively, }$$$$ \exists h>0, \exists r>0, \exists \ell^s>0, s=1,...,k,, \exists C=\max\{C_{r,s,\ell^s}>0;  s=1,...,k\},
$$
$$ |a_{j^1,...,j^s,p,i}|\leq C e^{-\sum_{s=1}^kM(\ell^s j^s)-M(rp^{1/n})-M(\mu_i^{1/q})},\;
(j^1,...,j^s,p,i)\in E_{k+2},
$$
\end{itemize}
Then   a solution exists if $(g1)$ holds.

 If $f\in\mathcal H^*(\prod_{s=1}^kX_s\times\mathbb R^n_+\times\mathbb R^m)$, then a solution is in the same space (hypoelliptic one),  if $(g2)$ holds.
\end{cor}

The proofs of these assertions are now clear.
  
\section*{Acknowledgements}
S. Pilipovi\'c is supported by the Serbian Academy of Sciences and Arts, project F10.
\DJ or\dj e Vu\v ckovi\'c research was  supported by the Science Fund of the Republic of
Serbia, $\#$GRANT No 2727, \emph{Global and local analysis of operators and
distributions} - GOALS.

\end{document}